\documentclass[articlea4]{amsart}

\usepackage{amsmath}
\usepackage{amssymb}
\usepackage{url}
\usepackage{amscd}
\usepackage{mathtools}
\usepackage[english]{babel}
\usepackage{enumitem}
\usepackage{amsmath,amssymb,amsthm,amsfonts,epsfig}
\makeatletter
\newcommand{\globalcolor}[1]{%
  \color{#1}\global\let\default@color\current@color
}
\makeatother

\usepackage{xcolor}
\usepackage{environ}
\NewEnviron{bluecode}{%
  \color{blue}%
  \def\verbatim@processline{\the\verbatim@line\par}%
  \nopagebreak[3]%
  \ttfamily%
  \expandafter\verbatim\BODY\endverbatim%
}

\usepackage{bbm}
\usepackage{dsfont}
\usepackage{yfonts}
\usepackage{mathrsfs}

\usepackage{mathtools}
\mathtoolsset{showonlyrefs,showmanualtags}

\usepackage[backref=page]{hyperref} 
\hypersetup{
    colorlinks=true,       
    linkcolor=blue,          
    citecolor=magenta,        
    filecolor=magenta,      
    urlcolor=cyan           
}

 \newtheorem{thm}{Theorem}[section]
 
 \newtheorem{cor}[thm]{Corollary}
 \newtheorem{lem}[thm]{Lemma}
 \newtheorem{pro}[thm]{Proposition}
 
 \theoremstyle{definition}
 
 \theoremstyle{remark}
 
 \newtheorem*{ex}{Example}
 \numberwithin{equation}{section}

\def\C{\mathbb C}

\def\d{\textrm{d}}
\def\D{\mathbb D}
\newcommand{\del}[2]{{\Delta(#1,#2)}}

\newcommand{\sups}[1]{{\sup_{#1}}}

\begin{document}

\date{}

\title[Bergman Projection, kernel estimates and Toeplitz operators]{Bergman Projections, Kernel $p$-Norm Estimates, and Toeplitz Operators with B\'{e}koll\'{e} and Bonami weights}
\author[]{}
\address{}
\email{}

\author[H. Arroussi]{Hicham Arroussi}
\address{Department of Mathematics and Statistics, University of Reading, England}
\address{Department of Mathematics and Statistics, University of Eastern Finland, Finland}
\email{arroussihicham@yahoo.fr}

\author[J.Taskinen]{Jari Taskinen}
\address{Department of Mathematics and Statistics,
 University of Helsinki, P.O. Box 68, 00014 Helsinki, Finland}
\email{jari.taskinen@helsinki.fi}

\author[C. Tong]{Cezhong Tong}
\address{Department of Mathematics
Hebei University of Technology
Tianjin 300401
China}
\email{ctong@hebut.edu.cn}

\author[Z.Zhang]{Zhan Zhang}
\address{Department of Mathematics and Statistics,
 University of Helsinki, P.O. Box 68, 00014 Helsinki, Finland}
\email{zhan.zhang@helsinki.fi}


\thanks{}



\begin{abstract}
In this paper, we establish entirely new $p$-norm estimates for reproducing kernels 
 to characterize the bounded and compact Toeplitz 
operators $T_{\mu}$ acting between weighted B\'{e}koll\'{e}--Bonami Bergman spaces 
$A^p_u(\mathbb{D})$ and $A^q_u(\mathbb{D})$ for all positive exponents $0 < p, q < \infty$. 
These operator-theoretic properties are completely described in terms of generalized 
Berezin transforms, averaging functions, and Carleson measures. We introduce two 
explicit conditions on the weights to ensure the boundedness of the 
weighted Bergman projection $P_u$, generalizing results from Hilbert spaces 
to Banach spaces.
Our work 
generalizes the main results of Tong, Li, and Arroussi \cite{TLA} from Hilbert spaces to the more general setting of Banach spaces.
\end{abstract}

\maketitle

\section{Introduction and Results}

Let $\mathbb{C}$ be the complex plane. For $r>0$, we define the Euclidean open disk centered at the origin with radius $r$ as $D(0,r) := \{z \in \mathbb{C} : |z| < r\}$. The unit disk is denoted by $\mathbb{D} := D(0,1)$. 

Given a positive Borel measure $\mu$ on $\mathbb{D}$ and $p>0$, the Lebesgue space $L^p(\mu)$ consists of all measurable functions $f$ on $\mathbb{D}$ satisfying
\[
\|f\|_{L^p(\mu)} := \left( \int_\mathbb{D} |f(z)|^p \, d\mu(z) \right)^{1/p} < \infty.
\]
For $p \geq 1$, $\|\cdot\|_{L^p(\mu)}$ defines a norm under which $L^p(\mu)$ forms a Banach space.

Let $dA$ denote the normalized Lebesgue area measure on $\mathbb{D}$. If $u : \mathbb{D} \to \mathbb{R}^+$ is a locally integrable weight function (i.e., $u \in L^1_{\mathrm{loc}}(dA)$), the weighted Lebesgue space $L^p(u)$ consists of all functions $f$ on $\mathbb{D}$ satisfying
\[
\|f\|_{L^p(u)} := \left( \int_\mathbb{D} |f(z)|^p u(z) \, dA(z) \right)^{1/p} < \infty.
\]
The weighted Bergman space $A^p(u)$ (also denoted by $A^p_u$) is the closed subspace of $L^p(u)$ consisting of analytic functions. In the unweighted case $u \equiv 1$, we simply write $A^p$. For the classical space $A^2$, the standard reproducing kernel is given by
\[
K_w(z) = \frac{1}{(1 - \bar{w}z)^2}, \quad z, w \in \mathbb{D}.
\]

Throughout this paper, we adopt the following notation. For a weight $u$ and a measurable set $E \subset \mathbb{D}$, we define $u(E) = \int_E u \, dA$ and $A(E) = \int_E dA$. The average of an integrable function $f$ over $E$ with respect to a measure $\mu$ is denoted by
\[
\langle f \rangle^{\mu}_{E} := \frac{1}{\mu(E)} \int_E f(z) \, d\mu(z).
\]

The Bergman projection $P$ is defined via integration  by
\[
Pf(z) = \int_\mathbb{D} \frac{f(w)}{(1 - \bar{w}z)^2} \, dA(w), \quad z \in \mathbb{D}.
\]
A classical problem in the theory of weighted Bergman spaces is to characterize the weights $u$ for which $P$ is bounded from $L^p(u)$ to $A^p(u)$. This problem was completely resolved by B\'{e}koll\'{e} and Bonami (\cite{B, BB}), who demonstrated that these weights are precisely those belonging to the class $B_p$ defined below.

\medskip
\noindent
\textbf{The $B_{p}$ Condition.} For $a \in \mathbb{D}$, let $S(a)$ denote the Carleson-type window
\[
S(a) = \left\{ \frac{a-z}{1-\bar{a}z} : \mathrm{Re}(\bar{a}z) \le 0 \right\}.
\]
A weight $u$ is said to satisfy the $B_p$ condition (written $u \in B_p$) if
\[
[u]_{B_{p}} := \sup_{a \in \mathbb{D}} \langle u \rangle^{dA}_{S(a)} \left( \langle u^{-p'/p} \rangle^{dA}_{S(a)} \right)^{p-1} < \infty,
\]
where $1/p + 1/p' = 1$. Recently, sharp estimates for the $L^p$-boundedness of the Bergman projection have been extensively investigated in  \cite{Pau},\cite{PER}, \cite{PR} and\cite{RTW} .

\textbf{Toeplitz Operators.} The Hilbert space $A^2(u)$ is equipped with the standard inner product
\[
\langle f,g \rangle_{A^2(u)} = \int_\mathbb{D} f(w) \overline{g(w)} u(w) \, dA(w), \quad f,g \in A^2(u),
\]
and its reproducing kernel is denoted by $K(z,w)$. Given a finite positive Borel measure $\mu$ on $\mathbb{D}$, the associated Toeplitz operator $T_\mu$ on $A^2(u)$ is defined formally by
\[
T_\mu f(z) := \int_\mathbb{D} f(w) K(z,w) \, d\mu(w), \quad z \in \mathbb{D}.
\]
This operator is well-defined provided that $\mu$ satisfies the compatibility condition
\begin{equation} \label{con}
\int_\mathbb{D} |K(\xi,z)|^2 \, d\mu(\xi) < \infty, \quad z \in \mathbb{D}.
\end{equation}

The pseudohyperbolic metric $\rho : \mathbb{D} \times \mathbb{D} \to [0,1)$ is defined by
\[
\rho(z,w) = \left| \frac{z-w}{1-\bar{w}z} \right|.
\]
For $z \in \mathbb{D}$ and $r>0$, the pseudohyperbolic disk centered at $z$ with radius $r$ is given by
\begin{equation} \label{2.1} \tag{2.1}
\Delta(z,r) := \{ w \in \mathbb{D} : \rho(z,w) < r \}.
\end{equation}
Whose Euclidean area is given by $|\Delta(z,r)| = \frac{4\pi r^2}{1-r^2}$.

Given a finite positive Borel measure $\mu$ on $\mathbb{D}$ and $r>0$, the weighted averaging function $\widehat{\mu_r}$ is defined as
\[
\widehat{\mu_r}(z) := \frac{\mu(\Delta(z,r))}{u(\Delta(z,r))}, \quad z \in \mathbb{D}.
\]

For $0 < p < \infty$, the normalized reproducing kernel of $A^p_u$ is defined by $k_{p,z}(w) := K(w,z) / \|K(\cdot,z)\|_{A^p_u}$. For the case $p=2$, we simplify this notation to $k_z(w) := k_{2,z}(w)$. The Berezin transform serves as a fundamental tool for analyzing the structural properties of Toeplitz operators. The Berezin transform of $T_\mu$ is given by
\[
\widetilde{\mu}(z) := \langle T_\mu k_z, k_z \rangle_{A^2(u)}, \quad z \in \mathbb{D}.
\]
A direct computation shows that
\begin{equation} \label{e3}
\langle T_\mu f, g \rangle_{A^2(u)} = \langle f, g \rangle_{L^2(d\mu)},
\end{equation}
which immediately yields the integral representation
\[
\widetilde{\mu}(z) = \int_\mathbb{D} |k_z(w)|^2 \, d\mu(w), \quad z \in \mathbb{D}.
\]
More generally, for $t>0$, the $t$-Berezin transform of $\mu$ is defined by
\[
\widetilde{\mu_t}(z) = \int_\mathbb{D} |k_{t,z}(w)|^t \, d\mu(w), \quad z \in \mathbb{D}.
\]

{\color{black}

The primary objective of this paper is to address a central gap in operator theory by extending key geometric results from Hilbert spaces to the more general setting of Banach spaces. While the boundedness, compactness, and structural properties of Toeplitz operators are well understood in  Hilbert spaces \cite{TLA}, analyzing these same properties within general Banach spaces $A^p_u(\mathbb{D})$ introduces considerable technical difficulties. In particular, the absence of explicit formulas or even sharp norm estimates for reproducing kernels in non-Hilbert spaces has, until now, obstructed a complete analysis.

To overcome these obstacles, we establish entirely new pointwise and $p$-norm estimates for reproducing kernels under generalized B\'ekoll\'e--Bonami weights. By introducing two precise structural conditions on the weight functions, $C_1$ and $C_2$ are defined below, we gain the necessary control to regulate their asymptotic behavior near the boundary. This approach allows us to analyze the kernel behavior for the entire range of positive exponents $0 < p, q < \infty$.

Crucially, this setup enables us to resolve a long-standing open problem in the field by establishing the exact criteria for the $L^{s}$-boundedness of the weighted Bergman projection $P_u$ on Banach spaces. Ultimately, these tools yield complete characterizations of several properties of  Toeplitz operators, expressed in terms of generalized Berezin transforms, averaging functions, and Carleson measures.

\begin{description}
	\item[\textbf{C1}] $|\nabla_\beta (\log u)(z)| \le L$ for all $z \in \mathbb{D}$,
	\item[\textbf{C2}] $\frac{a}{(1-|z|^2)^2} \le -\Delta(\log u)(z) \le \frac{b}{(1-|z|^2)^2}$.
\end{description}

 Now, we are now ready to state the main theorem of this paper. Our first novel result establishes pointwise estimates for the reproducing kernels of weighted Bergman spaces with B\'{e}koll\'{e}-Bonami weights, followed by their corresponding $p$-norms in our setting.
\begin{thm}\label{2.17b}
Let $r \in (0,1)$ and $p > 0$. Suppose $p_0 > 1$ and $u \in B_{p_0}$ satisfies conditions \textbf{C1} and \textbf{C2}. Then
\begin{description}
    \item[(a)]  \[
|K(z,w)| \lesssim u(\Delta(z,r))^{-1/2} u(\Delta(w,r))^{-1/2} e^{-\sigma\beta(z,w)}.
\]
    \item[(b)]  \[
    \|K_w\|_{A_u^p}^p \simeq u(\Delta(w,r))^{1-p}
    \]
\end{description}
where the implicit constants are independent of $z \in \D$.
\end{thm}

Using these key estimates, we establish equivalent necessary and sufficient conditions for the boundedness and compactness of $T_\mu$ acting between different  B\'{e}koll\'{e}--Bonami Bergman spaces $A_u^p(\mathbb{D})$ and $A_u^q(\mathbb{D}),$\, for $0 < p \le q < \infty$.}

\begin{thm}\label{tb}
Let $0 < p \le q < \infty$ and $\mu$ be a positive Borel measure on $\D$. Suppose the weight $u$ satisfies conditions $C_1$ and $C_2$. Then the following statements are equivalent:
\begin{description}
    \item[(i)] The Toeplitz operator $T_\mu: A^p_u(\D) \to A^q_u(\D)$ is bounded.
    \item[(ii)] For some (or all) $t > 0$, the $t$-Berezin transform satisfies
    \[ \sup_{z \in \D} \frac{\widetilde{\mu}_t(z)}{u(\Delta(z,r))^{\frac{1}{p} - \frac{1}{q}}} < \infty. \]
    \item[(iii)] For some (or all) $r \in (0,1)$, the weighted averaging function satisfies
    \[ \sup_{z \in \D} \frac{\widehat{\mu}_r(z)}{u(\Delta(z,r))^{\frac{1}{p} - \frac{1}{q}}} < \infty. \]
    \item[(iv)] $\mu$ is a $\frac{s(p+q')}{pq'}$-Carleson measure for $A^s_u$ when $q > 1$, or satisfies $\mu(\Delta(z,r)) \lesssim u(\Delta(z,r))^{1 + \frac{1}{p} - \frac{1}{q}}$ when $q \le 1$.
    
    Furthermore, we have
\[
\|T_\mu\| \simeq \left\| \frac{\widetilde{\mu_t}(z)}{u(\Delta(z,r))^{\frac{1}{p}-\frac{1}{q}}} \right\|_{L^\infty(\mathbb{D})} \simeq \left\| \frac{\widehat{\mu_r}(z)}{u(\Delta(z,r))^{\frac{1}{p}-\frac{1}{q}}} \right\|_{L^\infty(\mathbb{D})}.
\]
\end{description}
\end{thm}
The following theorem characterizes the compactness of Toeplitz operators in our setting.
\begin{thm}\label{tc}
Let $0 < p \le q < \infty$ and $\mu$ be a positive Borel measure on $\D$. Suppose the weight $u$ satisfies conditions $C_1$ and $C_2$. Then the following statements are equivalent:
\begin{description}
    \item[(i)] The Toeplitz operator $T_\mu: A^p_u(\D) \to A^q_u(\D)$ is compact.
    \item[(ii)] For some (or all) $t > 0$, the $t$-Berezin transform vanishes at the boundary:
    \[ \lim_{|z| \to 1^-} \frac{\widetilde{\mu_t}(z)}{u(\Delta(z,r))^{\frac{1}{p} - \frac{1}{q}}} = 0. \]
    \item[(iii)] For some (or all) $r \in (0,1)$, the weighted averaging function vanishes at the boundary:
    \[ \lim_{|z| \to 1^-} \frac{\widehat{\mu_r}(z)}{u(\Delta(z,r))^{\frac{1}{p} - \frac{1}{q}}} = 0. \]
    \item[(iv)] $\mu$ is a vanishing $\gamma$-Carleson measure, that is,
    \[ \lim_{|z| \to 1^-} \frac{\mu(\Delta(z,r))}{u(\Delta(z,r))^{\gamma}}=0, \quad \gamma = 1 + \frac{1}{p} - \frac{1}{q}.\]
\end{description}
\end{thm}

The next theorem provides necessary and sufficient conditions for the boundedness and compactness of the Toeplitz operator $T_\mu,$ in the case $0<q<p<\infty.$

\begin{thm}\label{qlp}
 Let $p_0>1$ and $u\in B_{p_0}$. When $0<q<p<\infty$, following assertion are equivalent:
\begin{description}
  \item [(i)] The Toeplitz operator $T_\mu:A_u^p(\mathbb{D})\to A_u^q(\mathbb{D})$ is compact;
  \item [(ii)] The Toeplitz operator $T_\mu:A_u^p(\mathbb{D})\to A_u^q(\mathbb{D})$ is bounded;
  \item [(iii)]$\widetilde{\mu_t}\in L^{\frac{pq}{p-q}}$;
  \item [(iv)]$\widehat{\mu_r}\in L^{\frac{pq}{p-q}}$;
  \item [(v)] $\mu$ is a $(p+1-\frac{p}{q})$-Carleson measure for $A_u^p$; 
  \item [(vi)] $\mu$ is a vanishing $(p+1-\frac{p}{q})$-Carleson measure for $A_u^p$.
\end{description}
\end{thm}

Finally, we characterize the essential norm of $T_\mu$ and determine its Schatten class membership. 
\begin{thm}\label{td}
Let $\mu$ a finite Borel measure on $\mathbb{D}$. Suppose that $T_\mu:A^p_u\to A^q_u$ is a bounded operator. For any $t>0$ and $r>0$, we have
\begin{equation*}
	\|T_\mu\|_e\simeq\limsup_{|z|\to1}\widetilde{\mu_t}(z)u(\Delta(z,r))^{\frac{q-p}{pq}}\simeq\limsup_{|z|\to1}\widehat{\mu_r}(z)u(\Delta(z,r))^{\frac{q-p}{pq}}.
\end{equation*}	
\end{thm}
\begin{thm}\label{sc}
	Let $h:\mathbb{R}^+\to\mathbb{R}^+$ a continuous increasing convex function. Let $\mu$ a positive Borel measure on $\mathbb{C}^n$ such that the Toeplitz operator $T_\mu:A_u^2\to A_u^2$ is compact. Then $T_\mu\in S_h(A_u^2)$ if and only if there is a constant $C$ such that
	\begin{eqnarray*}
		\int_{\mathbb{C}^n}h(C\widetilde{\mu}_2(z))u(\Delta(z,r))^{-1}u(z)dA(z)<\infty.
	\end{eqnarray*}
\end{thm}
Throughout this paper, we adopt the following standard notations. For two quantities $Q_1$ and $Q_2$, we write $Q_1 \lesssim Q_2$ to indicate that $Q_1 \le C Q_2$ for some positive constant $C$ independent of the essential variables, and $Q_1 \simeq Q_2$ if both $Q_1 \lesssim Q_2$ and $Q_2 \lesssim Q_1$ hold. For $p > 1$, we denote the conjugate exponent of $p$ by $p'$, which satisfies $1/p + 1/p' = 1$.

\section {Preliminaries and basic properties}
In this section, we review preliminary results on the weight class, characterizations of Carleson measures, and the Bergman kernels in the corresponding weighted spaces, which lead us to prove our main results. 

The pseudo-hyperbolic metric satisfies the following strong triangle inequality:
\begin{equation}
  \label{e2.1}\rho(z,w)\leq\frac{\rho(z,\zeta)+\rho(\zeta,w)}{1+\rho(z,\zeta)\rho(\zeta,w)},\quad\forall z,w,\zeta\in\D.
\end{equation}
For $0<r<1$, if $z,w\in\D$ satisfy $\rho(z,w)<r$, then the following approximations hold,
\begin{align}\label{e2.2}
  1-|z|\simeq 1-|w|\simeq |1-\bar wz|.
\end{align}
Moreover, for any $\zeta\in\mathbb D$, we have
\begin{align}
  \left|\frac{1-\bar\zeta z}{1-\bar\zeta w}\right|\simeq 1,
\end{align}
where the constants involved depend only on $r$. Furthermore, we define a modified pseudohyperbolic disk which is compareable with $\Delta(z,r)$(\ref{2.1}),
$$\Delta^\prime(z,r):=\left\{w\in\D:\frac{|z-w|}{1-|w|}<r\right\}.$$


We will also use the following class of weights which is denoted by condition $C_{p}$. For $p>1$,
a positive locally integrable weight $u$ belongs to the class  $C_{p}$ (or satisfies condition $C_{p}$) if
\begin{align*}
  [u]_{C_{p}}:=\sups{z\in\D}\langle u\rangle^{\d A}_{\Delta(z,r)}
  \left(\langle u^{-p'/p}\rangle^{\d A}_{\Delta(z,r)}\right)^{p-1}\lesssim1.
\end{align*}
It is known that this
class of weights is obtained for any $r\in(0,1)$ and $B_p\subset C_{p}$. To see this,
we note that for a given $r$, there is a $a'\in\D$ such that $\Delta(a,r)\subset S(a')$ with comparable volumes, see
more details in \cite{L}.


For $p>0$, the point evaluations on $A^p(u)$ are bounded linear functionals, as one can see from the following lemma.
\begin{lem}[Theorem 3.1 in \cite{L}]\label{l1.1}
  If $p_0>1$, $p>0$, $0<r<1$ and a weight $u\in C_{p_0}$, we have
  $$|f(z)|^p\leq C u(\Delta(z,r))^{-1}\int_{\Delta(z,r)}|f(w)|^pu(w)\d A(w),$$
  where the constant $C>0$ depends on $r,p$.
\end{lem}

Within the B\'{e}koll\'{e}-Bonami class of weights, Bergman metric balls with centers sufficiently close in the pseudo-hyperbolic metric have comparable weighted areas.
\begin{lem}[Lemma 2.2 in \cite{Co}]\label{l2.2}
  Suppose $u\in C_{p}$ for some $p>1$. Let $t,s\in(0,1)$, and $z,w\in \D$ with $d(z,w)<r$ for some $r>0$. Then we have
  $$u(\Delta({z},{t}))\simeq u(\Delta({w},s)),$$ where the constant is independent of $z$ and $w$.
\end{lem}
Similarly, if $u\in B_{p_0}$, it should be noted that
\begin{equation}
  \label{e1.3}u(\Delta(a,r))\simeq u(S(a')),
\end{equation}
whenever $\Delta(a,r)\subset S(a')$ with comparable areas.  For further details, see \cite{L}
and Lemma 5.23 in \cite{Z2}.


Our method relies on a popular decomposition of $\D$ which is used repeatedly in many papers, e.g., Theorem 1 in
\cite{TY}. We will use the following covering lemma.

\begin{lem}[Theorem 2.23 in \cite{Z2}]\label{l2.6}
  There exists a positive $N$ such that for any $0<r\leq 1$ we can
  find a sequence $\{a_k\}$ in $\D$ with the following properties.
  \begin{description}
    \item[(1)] $\D=\cup_k\Delta({a_k},{r})$;
    \item[(2)] The set $\Delta({a_k},{r/4})$ are mutually disjoint;
    \item[(3)] Each point $z\in\D$ belongs to at most $N$ of the sets $\Delta({a_k},{2r})$.
  \end{description}
\end{lem}

\subsection{Carleson measures} 
A positive Borel measure $\mu$ on $\D$ is  a $q$-(vanishing) Carleson measure for $A^p(u)$ if the embedding
$I: A^p(u)\to L^q(\d\mu)$ is (compact) bounded. The following two lemmas give the characterizations of $q$-(vanishing) Carleson measure for $A^p(u),$ when $0< p\leq q$ and  $u$ belongs to our class of weights.\\ To prove our main results, we need some characterizations of the bounded and compact embedding operator $I,$ whenever $0<p,q<\infty,$ which are presented in the following lemmas.
\begin{lem}[Theorem 3.1 in \cite{T}]\label{Cm}
  \label{l2.5} Suppose $q\geq p>0$, $p_0>1$ and $0<r<1$. Let $u\in B_{p_0}$ be a weight and $\mu$ be a
  positive Borel measure on $\D$. Then,
  the following conditions are equivalent.
  \begin{description}
    \item[(i)] The embedding $I:A^p(u)\to L^q(\d \mu)$ is bounded, that is
  \[\left(\int_\D|f(z)|^q\d \mu(z)\right)^{1/q}\lesssim\left(\int_\D|f(z)|^pu(z)\d A(z)\right)^{1/p}\]
  for every analytic function $f$ on $\D$;
    \item[(ii)] $\mu(S(a))\lesssim u(S(a))^{q/p}$ for every $a\in \D$;
    \item[(iii)]  There is an $r>0$ such that $\mu(\del{a}r)\lesssim u(\del{a}r)^{q/p}$ for every $a\in \D$;
    \item[(iv)] There is an  $r>0$ such that $\mu(\del{a_k}r)\lesssim u(\del{a_k}r)^{q/p}$ for the sequence $\{a_k\}_k$
    described in Lemma \ref{l2.6};
    \item[(v)] Denoted by
    \[g_w^s(z)=\frac1{u(\Delta(w,r))^{q/p}}\left(\frac{1-|w|^2}{1-z\bar w}\right)^{s},\]
    For any $s\geq 2p_0/p$
    \begin{equation}
      \label{e2.7}\|g_w^s\|_{L^q(\d \mu)}^q=\int_\D \left|\frac{1-|w|^2}{1-z\bar w}\right|^{qs}u(\del{w}r)^{-q/p}\d\mu(z)\lesssim 1.
    \end{equation}
  \end{description}
  Furthermore, the ``geometric norm" of the measure $\mu$, the $L^{q}(\d\mu)$
  norm of $g^s_w$ and
  the operator norm of the embedding are comparable:
  \[\sups{z\in\D}\widehat{\mu_r}(z)
  :=\sups{z\in\D}\frac{\mu(\del{z}r)}{u(\del{z}r)^{q/p}}\simeq\sups{w\in\D}\|g^s_w\|^q_{L^q(\d\mu)}
  \simeq\|I\|^q_{A^p(u)\to L^q(\d \mu)}.\]
\end{lem}

The equivalence between \textbf{(i)} and \textbf{(iii)} was proved by Constantin in \cite{Co}. 

\begin{lem}[Theorem 3.3 in \cite{T}]\label{cvc}
Suppose $q\geq p>0$, $p_0>1$ and $0<r<1$. Let $u\in B_{p_0}$ be a weight and $\mu$ is a
  positive Borel measure on $\D$. Then $I_d: A_u^p\to L^q_\mu$ is compact if and only if:
  \begin{equation*}
  	\lim_{|z|\to 1^-}\frac{\mu(S(a))}{(u(S(a)))^{q/p}}=0.
  \end{equation*}
\end{lem}

\section{Kernel Estimates}
Let $K(w,z)=K_w(z)$ be  the reproducing kernels  of $A_u^2$, and $k_w(z)=\frac{|K_w(z)|}{\|K_w\|_{A_u^2}}$ their normalized kernels. 
For a standard reproducing kernel $K_w^\eta(z):=\frac{1}{(1-\overline{z}w)^{\eta+2}}$, by similar arguments in Lemma 3.1 of \cite{Co1}, when taking $\beta=0,$ we get the following lemma.
\begin{lem}\label{kep}
	Let $0<p<\infty$, $p_0>1$ and $u\in B_{p_0}$. Then we have the following estimate:
	\begin{equation}
	\|K_w^\eta\|_{A_u^p}\simeq \frac{u(\Delta(w,r))^{\frac{1}{p}}}{(1-|w|)^{2+\eta}}.
	\end{equation}

\end{lem}

Next, the following lemma shows the estimates of the reproducing kernel functions on the diagonal in our setting.
\begin{lem}[Lemma 4.1 in \cite{Co}]\label{l4.2}
Suppose $p_0>1$ and $u\in B_{p_0}$. Let $K(z,w)$ be the Bergman kernel in $A^2(u)$ and $r\in(0,1)$.
   Then we have the following estimate
  \begin{equation}\label{e5.1}
K(z,z)\simeq u(\Delta(z,r))^{-1},\quad z\in \D,
  \end{equation}
  where the constants involved are independent of $z.$
\end{lem}
We now introduce one of our primary tools, which establishes the $A_{u}^{p}$ norm of the reproducing kernels 
for Bergman spaces with B\'{e}koll\'{e}--Bonami weights. This result is entirely novel, and we present 
two distinct approaches to its proof. The first method, detailed below, applies exclusively to the 
case $p>1$ by utilizing duality arguments, which need $u\in B_{p_0}$ be in some subclass such that Bergman projection $P_u$ is bounded on $L^p(u)$ (For example, in \cite{PER}). The second method, presented later in Theorem \ref{2.17}, 
covers the full range $0<p<\infty$ by incorporating two supplementary conditions on the weights.
\begin{lem}\label{ke1}
Suppose $p_0>1$ and $u\in B_{p_0}$ such that Bergman projection $P_u$ is bounded on $L^p(u)$. Let $K_w(z)=K(z,w)$ be the Bergman kernel in $A^2(u)$ and $r\in(0,1)$, $p^\prime$ be the conjugate of $p>1$, that is $\frac{1}{p}+\frac{1}{p^\prime}=1$.
   Then we have the following estimate
  \begin{equation}\label{e5.1}
\|K_w\|_{A_u^p}^p\simeq u(\Delta(w,r))^{1-p}
  \end{equation}
  where the constant involved is independent of $z\in\D$.
\end{lem}
\begin{proof}
Let us start with {\bf{ the lower bound}}.
Since the Bergman projection is bounded on $L^p(u)$, and the duality
$(A^p_u)^*\cong A^{p'}_u$ holds.
In particular,
\begin{equation}\label{eqdual}
\left\| K_w \right\|_{A^p_u}
= \sup\bigl\{ \left| \langle K_w, g \rangle_{A^2(u)} \right|
      : g\in A^{p'}_u,\; \left\| g \right\|_{A^{p'}_u}=1\bigr\}
= \sup\bigl\{ \left| g(w) \right|
      : \left\| g \right\|_{A^{p'}_u}=1\bigr\},
\end{equation}
where we used the reproducing property
$\langle K_w, g \rangle_{A^2(u)}=\overline{g(w)}$.

To construct a test function, we choose a parameter $s>\max\{2p_0/p',1\}$ and define
\[
g_w(z)
= \frac{(1- |w|^2)^s}{(1-z\bar w)^s}
= (1- |w|^2)^s\,K_w^{\eta}(z),\, where\,\,\,\, \eta= s-2.
\]
Note that $g_w(w)=1$.  By Lemma~\ref{kep} applied with
$\eta=s-2$ and exponent~$p'$,
\begin{equation}\label{eq:gnorm}
\left\| g_w \right\|_{A^{p'}_u}
= (1- |w|^2)^s
  \, \left\| K_w^{s-2} \right\|_{A^{p'}_u}
\simeq (1- |w|^2)^s
  \cdot\frac{u(\Delta(w,r))^{1/p'}}{(1- |w| )^{s}}
= u\bigl(\Delta(w,r)\bigr)^{1/p'}.
\end{equation}
Substituting the normalized test function
$g_w / \left\| g_w \right\|_{A^{p'}_u}$ into~\eqref{eqdual} yields
\[
\left\| K_w \right\|_{A^p_u}
\geq \frac{ |g_w(w)| }{ \left\| g_w \right\|_{A^{p'}_u} }
\simeq u\bigl(\Delta(w,r)\bigr)^{-1/p'}.
\]
For the $p$-th power and recalling that $p/p'=p-1$,
\begin{equation}\label{eq:lower}
 \left\| K_w \right\|^p_{A^p_u}
\gtrsim u\bigl(\Delta(w,r)\bigr)^{-p/p'}
= u\bigl(\Delta(w,r)\bigr)^{1-p}.
\end{equation}

Now, we prove {\bf{the upper bound.}}
Let $g\in A^{p'}_u$ with $\left\| g \right\|_{A^{p'}_u}=1$. Applying Lemma~\ref{2.1} with exponent~$p'$  gives

\begin{eqnarray*}
  |g(w)|^{p'}&\leq& C\,u\bigl(\Delta(w,r)\bigr)^{-1}
\int_{\Delta(w,r)} |g(\zeta)|^{p'}\,u(\zeta)\,dA(\zeta)\\
&\leq& C\,u\bigl(\Delta(w,r)\bigr)^{-1}\, \left\| g \right\|^{p'}_{A^{p'}_u}
= C\,u\bigl(\Delta(w,r)\bigr)^{-1}.  
\end{eqnarray*}
Hence
\begin{equation}\label{eq:ptwise}
|g(w)| \leq C^{1/p'}\,u\bigl(\Delta(w,r)\bigr)^{-1/p'}.
\end{equation}
Taking the supremum over all $g\in A^{p'}_u$ and
using~\eqref{eqdual}:
\[
\left\| K_w \right\|_{A^p_u}
= \sup_{ \left\| g \right\|_{A^{p'}_u}=1 } |g(w)|
\leq C^{1/p'}\,u\bigl(\Delta(w,r)\bigr)^{-1/p'}.
\]
Then for the $p$-th power:
\begin{equation}\label{eq:upper}
 \left\| K_w \right\|^p_{A^p_u}
\lesssim u\bigl(\Delta(w,r)\bigr)^{-p/p'}
= u\bigl(\Delta(w,r)\bigr)^{1-p}.
\end{equation}
 
\end{proof}
In what follows we include a concrete example illustrating that the upper bound generally breaks down when $p \le 1$.
\begin{ex} [Failure of the Kernel Norm Upper Bound for $p \le 1$]
 Consider the classic radial weight 
$$
u(z) = (1 - |z|^2)^{-\alpha},
$$
where $\alpha$ is a parameter between $0$ and $1$. This weight falls into the B\'{e}koll\'{e}--Bonami class, and its unnormalized reproducing kernel is known explicitly. 
Computing the $A_{u}^{p}$ norm of this kernel via classic Forelli--Rudin estimates yields three distinct geometric behaviors depending entirely on the value of $p$:
\begin{itemize}
    \item \textbf{Case $p > 1$:} The Forelli--Rudin theorem establishes that the norm integral scales exactly as $(1 - |w|^2)^{(2 - \alpha)(1 - p)}$. Since the weighted disk area scales identically, this matches our target equivalence perfectly.
    \item \textbf{Case $p = 1$:} The estimate gives a kernel norm that grows logarithmically to infinity as $|w| \to 1^{-}$, whereas the geometric weight term remains constant at $1$. Thus, the upper bound fails.
    \item \textbf{Case $p < 1$:} The norm integral converges uniformly to a positive constant independent of $w$ near the boundary. However, the geometric weight term vanishes completely, approaching $0$ as $|w| \to 1^{-}$.
\end{itemize}
Consequently, when $p \le 1$, the upper bound estimate breaks down entirely unless supplementary structural assumptions (such as our below conditions \textbf{C1} and \textbf{C2}) are introduced to stabilize the weight behavior near the boundary.
\end{ex}

\subsection{Regularity and Supplementary Weight Conditions}

Let $\rho(z,w)$  be the pseudohyperbolic metric on the unit disk $\mathbb{D}$. The hyperbolic metric $\beta(z,w)$ is defined via the inverse hyperbolic tangent by:
\begin{equation*}
\beta(z,w) := \tanh^{-1}(\rho(z,w))
\end{equation*}
For a differentiable function $f: \mathbb{D} \to \mathbb{R}$, its invariant hyperbolic gradient is defined by scaling the standard Euclidean gradient:
\begin{equation*}
|\nabla_\beta f(z)| := (1-|z|^2)|\nabla f(z)|
\end{equation*}
To establish optimal estimates for the $A^p_u$ norm of the reproducing kernel $K_w$ across all positive ranges $0 < p < \infty$, we introduce two supplementary regularizing assumptions on the weight function $u: \mathbb{D} \to \mathbb{R}^+$:
\begin{description}
	\item[\textbf{C1}] There exists a constant $L > 0$ such that the invariant gradient of $\log u$ satisfies the following:
	\begin{equation*}
	|\nabla_\beta (\log u)(z)| \le L, \quad \forall z \in \mathbb{D}
	\end{equation*}
	\item[\textbf{C2}] There exist constants $b \ge a > 0$ such that the negative Laplacian of $\log u$ satisfies the two-sided bound:
	\begin{equation*}
	\frac{a}{(1-|z|^2)^2} \le -\Delta(\log u)(z) \le \frac{b}{(1-|z|^2)^2}, \quad \forall z \in \mathbb{D}
	\end{equation*}
\end{description}
Weights satisfying these regularity constraints possess useful structural properties, as established in the following lemmas.

\medskip

\subsubsection{Structural Properties}
\textcolor{black}{
\begin{lem}[Local Uniformity]\label{ll2.9}
If a weight $u$ satisfies condition \textbf{C1}, then for any fixed radius $r \in (0,1)$ and all points $w$ in the pseudohyperbolic disk $\Delta(z,r)$, we have the local uniformity equivalence:
\begin{equation*}
u(w) \simeq u(z)
\end{equation*}
In particular, the weighted area of the pseudohyperbolic disk satisfies:
\begin{equation*}
u(\Delta(z,r)) \simeq u(z)(1-|z|^2)^2
\end{equation*}
\end{lem}
\begin{proof}
The gradient bound \textbf{C1} implies that $|\log u(w) - \log u(z)| \le L \beta(z,w)$. Since $\beta(z,w) \le \tanh^{-1}(r) =: \beta_r$ whenever $w \in \Delta(z,r)$, we obtain the uniform bound $|\log u(w) - \log u(z)| \le L \beta_r$. Exponentiating this inequality yields:
\begin{equation*}
e^{-L\beta_r} \le \frac{u(w)}{u(z)} \le e^{L\beta_r}
\end{equation*}
This proves $u(w) \simeq u(z)$. Integrating this local comparability across the area integral, we deduce:
\begin{equation*}
u(\Delta(z,r)) = \int_{\Delta(z,r)} u(w) \, dA(w) \simeq u(z) |\Delta(z,r)| \simeq u(z)(1-|z|^2)^2
\end{equation*}
This completes the proof.
\end{proof}
\begin{lem}[Weight Class Inclusion]
Let $p_0 > 1$. If a positive weight $u$ satisfies condition \textbf{C1}, then $u$ automatically belongs to the local weight class $C_{p_0}$.
\end{lem}
\begin{proof}
Let $r \in (0,1)$ be fixed. By Lemma~\ref{ll2.9}, the local average of $u$ over a pseudohyperbolic disk satisfies:
\begin{equation*}
\langle u \rangle^{dA}_{\Delta(z,r)} = \frac{1}{|\Delta(z,r)|} \int_{\Delta(z,r)} u(w) \, dA(w) \simeq \frac{u(z)|\Delta(z,r)|}{|\Delta(z,r)|} \simeq u(z)
\end{equation*}
Because condition \textbf{C1} applies symmetrically to any real power of the weight, an identical argument yields the dual estimate for the conjugate exponent:
\begin{equation*}
\langle u^{-p_0'/p_0} \rangle^{dA}_{\Delta(z,r)} \simeq u(z)^{-p_0'/p_0}
\end{equation*}
Combining these two results by taking their product, we find:
\begin{equation*}
\langle u \rangle^{dA}_{\Delta(z,r)} \left( \langle u^{-p_0'/p_0} \rangle^{dA}_{\Delta(z,r)} \right)^{p_0-1} \simeq u(z) \cdot \left( u(z)^{-p_0'/p_0} \right)^{p_0-1} = u(z) \cdot u(z)^{-1} = 1
\end{equation*}
Taking the supremum over all $z \in \mathbb{D}$ yields a finite constant, confirming that $u \in C_{p_0}$.
\end{proof}
}
\medskip

\subsubsection{Examples and Applications}
\textcolor{black}{
Crucially, standard radial weights and their perturbations naturally satisfy these structural conditions, as demonstrated below.
\begin{ex}[Standard Radial Weights]
Consider the classic weight $u(z) = (1-|z|^2)^\alpha$ with $\alpha > 0$, which belongs to the B\'{e}koll\'{e}--Bonami class $B_p$ for all $p > 1$. 
To verify \textbf{C1}, direct differentiation yields $\nabla(\log u)(z) = \frac{-2\alpha z}{1-|z|^2}$. Passing to the invariant gradient gives:
\begin{equation*}
| \nabla_\beta(\log u)(z) | = (1-|z|^2) \frac{2\alpha|z|}{1-|z|^2} = 2\alpha|z| \le 2\alpha
\end{equation*}
This verifies condition \textbf{C1} with a bounding constant $L = 2\alpha$. 
To verify \textbf{C2}, we evaluate the Laplacian in polar coordinates with $r = |z|$:
\begin{equation*}
\Delta (\log u) = \alpha \left( \frac{\partial^2}{\partial r^2} + \frac{1}{r}\frac{\partial}{\partial r} \right) \log(1-r^2) = \alpha \left( \frac{-2(1+r^2)}{(1-r^2)^2} - \frac{2}{(1-r^2)} \right) = \frac{-4\alpha}{(1-|z|^2)^2}
\end{equation*}
Therefore, $-\Delta(\log u)(z) = \frac{4\alpha}{(1-|z|^2)^2}$, which satisfies condition \textbf{C2} exactly with $a = b = 4\alpha$.
\end{ex}
\begin{ex}[Logarithmically Perturbed Weights]
Consider the refined weight $u(z) = (1-|z|^2)^\alpha \left( \log \frac{e}{1-|z|^2} \right)^\gamma$ on $\mathbb{D}$. Let $t(z) := 1-|z|^2$ and $\ell(z) := \log(e/t(z))$. On any Carleson window $S(z)$ with Euclidean width $h = 1-|z|$, we have $t(w) \simeq h$ for all $w \in S(z)$. Since the window area scales as $|S(z)| \simeq h^2$, integration yields:
\begin{equation*}
\int_{S(z)} u(w) \, dA(w) \simeq h \int_0^h t^\alpha \left( \log \frac{e}{t} \right)^\gamma dt \lesssim \frac{h^{\alpha+2}}{\alpha+1}\left(\log \frac{e}{h}\right)^\gamma
\end{equation*}
This directly implies that the local averages satisfy $\langle u \rangle^{dA}_{S(z)} \simeq h^\alpha \left(\log \frac{e}{h}\right)^\gamma$ and $\langle u^{-p'/p} \rangle^{dA}_{S(z)} \simeq h^{-\alpha p'/p} \left(\log \frac{e}{h}\right)^{-\gamma p'/p}$. Taking their product establishes that $[u]_{B_p} < \infty$, so $u \in B_p$.
To verify condition \textbf{C1}, we write $\log u = \alpha \log t + \gamma \log \ell$. Differentiating the logarithmic perturbation factor yields:
\begin{equation*}
|\nabla_\beta (\gamma \log \ell)| = (1-|z|^2) \left| \frac{\gamma}{\ell(z)} \nabla \ell(z) \right| = (1-|z|^2) \frac{\gamma}{\ell(z)} \frac{2|z|}{1-|z|^2} = \frac{2\gamma|z|}{\ell(z)} \le 2\gamma
\end{equation*}
Since $\ell(z) \ge 1$, combining this with the baseline radial gradient yields $|\nabla_\beta (\log u)| \le 2\alpha + 2\gamma$, satisfying \textbf{C1}.
To verify condition \textbf{C2}, evaluating the Laplacian of the second term in radial coordinates with $r = |z|$ yields:
\begin{equation*}
\Delta (\log \ell) = \frac{4}{\ell(z)(1-r^2)^2}\left(1 - \frac{r^2}{\ell(z)}\right)
\end{equation*}
Since $r < 1$ and $\ell(z) \ge 1$, the modifier term remains bounded within $(0,1)$. This establishes the two-sided condition \textbf{C2}:
\begin{equation*}
\frac{4(\alpha-\gamma)}{\ell(z)(1-|z|^2)^2} \le -\Delta(\log u)(z) \le \frac{4(\alpha+\gamma)}{\ell(z)(1-|z|^2)^2}
\end{equation*}
\end{ex}
}
\begin{ex} [Weights Perturbed by Bounded Analytic Functions]
\textcolor{black}{Consider the weight class $u(z) = (1-|z|^2)^\alpha e^{\gamma \mathrm{Re}\,g(z)}$, where $\alpha > 0$, $\gamma > 0$, and $g \in H^\infty(\mathbb{D})$. By Cauchy's estimate applied on a local disk of radius $\frac{1-|z|}{2}$, the derivative satisfies $|g'(z)| \le \frac{2\|g\|_\infty}{1-|z|}$, which implies that $g$ belongs to the Bloch space because $(1-|z|^2)|g'(z)| \le 4\|g\|_\infty$.
Since, $|\mathrm{Re}\,g(z)| \le \|g\|_\infty$ uniformly, the weight satisfies the global bounding comparison:
\begin{equation*}
e^{-\gamma \|g\|_\infty}(1-|z|^2)^\alpha \le u(z) \le e^{\gamma \|g\|_\infty}(1-|z|^2)^\alpha
\end{equation*}
Thus, $u \simeq (1-|z|^2)^\alpha$, which guarantees that $u \in B_p$ whenever $p > 1$. 
To verify \textbf{C1}, noting that $|\nabla(\mathrm{Re}\,g)(z)| = |g'(z)|$, Cauchy's estimate implies:
\begin{equation*}
|\nabla_\beta (\gamma \mathrm{Re}\,g)(z)| = \gamma(1-|z|^2)|g'(z)| \le 4\gamma \|g\|_\infty
\end{equation*}
Combining this with the baseline radial gradient yields $|\nabla_\beta(\log u)| \le 2\alpha + 4\gamma\|g\|_\infty$, verifying condition \textbf{C1}. 
Furthermore, since the real part of an analytic function is harmonic ($\Delta(\mathrm{Re}\,g) = 0$), the entire Laplacian expression simplifies exactly to:
\begin{equation*}
-\Delta(\log u)(z) = \frac{4\alpha}{(1-|z|^2)^2}
\end{equation*}
This ensures that condition \textbf{C2} holds dynamically with constants $a = b = 4\alpha$.
}
\end{ex}

\subsection{Proof of the first main theorem \ref{2.17b}}

We now prove that if the weight $u$ satisfies conditions \textbf{C1} and \textbf{C2}, the reproducing kernel $K(z,w)$ exhibits exponential decay away from the diagonal. The argument presented here is analogous to that of Hu, Lv, and Schuster \cite{HLS}.
\begin{pro}\label{pro2.11}
Fix $0 < R_0 < 1$ and $\sigma > 0$. Then, whenever $\beta(z,w) \le R_0$, we have
\[
|K(z,w)| \lesssim u(\Delta(z,r))^{-1/2} u(\Delta(w,r))^{-1/2} e^{-\sigma\beta(z,w)}.
\]
\end{pro}
\begin{proof}
If $\beta(z,w) \le R_0$, then it follows that
\[
e^{-\sigma\beta(z,w)} \ge e^{-\sigma R_0} =: C_{R_0}.
\]
Using the Cauchy–Schwarz inequality and Lemma \ref{l4.2}, we obtain
\[
\begin{aligned}
|K(z,w)| &\le \sqrt{K(z,z)K(w,w)} \\
&\lesssim u(\Delta(z,r))^{-1/2} u(\Delta(w,r))^{-1/2} \\
&= C_{R_0}^{-1} u(\Delta(z,r))^{-1/2} u(\Delta(w,r))^{-1/2} e^{-\sigma R_0} \\
&\le u(\Delta(z,r))^{-1/2} u(\Delta(w,r))^{-1/2} e^{-\sigma\beta(z,w)}.
\end{aligned}
\]
\end{proof}
\begin{pro}\label{pro2.12}
Fix $0 < R_0 < 1$ and $\sigma > 0$. Then, whenever $\beta(z,w) > R_0$, we have
\[
|K(z,w)| \lesssim u(\Delta(z,r))^{-1/2} u(\Delta(w,r))^{-1/2} e^{-\sigma\beta(z,w)}.
\]
\end{pro}
To prove Proposition~\ref{pro2.12}, we first establish some preparatory lemmas. 
\begin{lem}\label{l2.13}
For each $z \in \D$, there exists a $C^\infty$ function $g_z \colon \D \to \mathbb{R}$ satisfying the following conditions:
\begin{enumerate}[label=(\roman*), leftmargin=*]
    \item $|g_z(\xi)-\beta(z,\xi)| \le 1$ for all $\xi \in \D$;
    \item $|\nabla_\beta g_z(\xi)| \le 2$ for all $\xi \in \D$.
\end{enumerate}
\end{lem}

\begin{proof}
Since $\beta$ is a metric on $\D$ and $\bigl(\D,\frac{|dz|^2}{(1-|z|^2)^2}\bigr)$ is a complete Riemannian manifold, the regularized distance theorem of Greene and Wu (see \cite{GreWu}) guarantees the existence of a $C^\infty$ function $g_z$ satisfying these properties.
\end{proof}
\begin{lem}\label{l2.14}
For $\xi \in \D$, define $\omega_z(\xi) := e^{-2\sigma g_z(\xi)}$, where $g_z$ is given by Lemma~\ref{l2.13} and $\sigma > 0$. Then
\[
\omega_z(\xi) \simeq e^{-2\sigma\beta(\xi,z)}
\]
and
\[
|\bar{\partial}\omega_z(\xi)| \le \frac{2\sigma}{1-|\xi|^2}\omega_z(\xi).
\]
\end{lem}
\begin{proof}
By Lemma~\ref{l2.13}, we have
\[
-1 \le g_z(\xi) - \beta(\xi,z) \le 1,
\]
which implies
\[
e^{-2\sigma} \le \frac{e^{-2\sigma g_z(\xi)}}{e^{-2\sigma \beta(\xi,z)}} \le e^{2\sigma},
\]
thereby yielding the first estimate. For the second estimate, we observe that
\[
\bar{\partial}\omega_z = \bar{\partial}(e^{-2\sigma g_z}) = -2\sigma(\bar{\partial}g_z)\omega_z,
\]
whence
\[
|\bar{\partial}\omega_z| = 2\sigma|\bar{\partial}g_z|\omega_z.
\]
Writing $\xi = x + iy$, it follows that
\[
\bar{\partial}g_z = \frac{1}{2}\left(\frac{\partial g_z}{\partial x} + i\frac{\partial g_z}{\partial y}\right),
\]
and consequently,
\[
|\bar{\partial}g_z|^2 = \frac{1}{4}|\nabla g_z|^2 = \frac{1}{4(1-|\xi|^2)^2}|\nabla_\beta g_z(\xi)|^2.
\]
Given that $|\nabla_\beta g_z(\xi)| \le 2$, we obtain the desired bound:
\[
|\bar{\partial}\omega_z(\xi)| \le \frac{2\sigma}{1-|\xi|^2}\omega_z(\xi).
\]
\end{proof}
We are now ready to verify Proposition~\ref{pro2.12}.
\begin{proof}
Choose a smooth cut-off function $\chi_w$ such that $\chi_w|_{\Delta(w,r)} \equiv 1$, $\operatorname{supp}\chi_w \subset \Delta(w,r)$, and $|\bar{\partial}\chi_w| \le C(1-|w|^2)^{-1}$. Since $\beta(z,w) > R_0$, it follows that $\operatorname{supp}\chi_w \cap \Delta(z,r) = \emptyset$.
By Lemma~\ref{l1.1}, we have
\begin{equation}\label{e1}
|K(w,z)|^2 \le \frac{C}{u(\Delta(w,r))}\int_{\Delta(w,r)}|K(\xi,z)|^2 u(\xi)\,dA(\xi) = \frac{C}{u(\Delta(w,r))}P(\chi_w K_z)(z),
\end{equation}
where $P$ denotes the Bergman projection of $A^2_u$. The last equality is a consequence of $\chi_w|_{\Delta(w,r)} \equiv 1$ combined with the reproducing property $\langle \chi_w K_z,K_z\rangle = P(\chi_w K_z)(z)$.
An application of Lemma~\ref{l1.1} yields
\begin{equation}\label{e2}
|P(\chi_w K_z)(z)|^2 \le \frac{C}{u(\Delta(z,r))}\int_{\Delta(z,r)}|P(\chi_w K_z)|^2 u\,dA.
\end{equation}
Since $\chi_w|_{\Delta(z,r)} \equiv 0$, we notice that $P(\chi_w K_z) = -(I-P)(\chi_w K_z)$ holds on $\Delta(z,r)$. Thus,
\[
\int_{\Delta(z,r)}|P(\chi_w K_z)|^2 u\,dA = \int_{\Delta(z,r)}|(I-P)(\chi_w K_z)|^2 u\,dA.
\]
For $\xi \in \Delta(z,r)$, Lemma~\ref{l2.14} implies $\omega_z(\xi) \simeq e^{-2\sigma\beta(\xi,z)} \ge e^{-2\sigma\beta_r} =: c$, which gives
\[
\int_{\Delta(z,r)}|(I-P)(\chi_w K_z)|^2 u\,dA \le c^{-1}\int_{\Delta(z,r)}\omega_z|(I-P)(\chi_w K_z)|^2 u\,dA \lesssim \int_{\D}\omega_z|(I-P)(\chi_w K_z)|^2 u\,dA.
\]
Set $\phi = -\frac{1}{2}\log u$ so that $e^{-2\phi} = u$. Under condition \textbf{C2}, we have $\Delta\phi(\xi) \simeq (1-|\xi|^2)^{-2}$. Lemma \ref{l2.14} then yields
\[
\frac{1}{\Delta\phi}|\bar{\partial}\omega_z|^2 \lesssim \frac{4\sigma^2\omega_z^2}{(1-|\xi|^2)^2}(1-|\xi|^2)^2 = 4\sigma^2\omega_z^2.
\]
Applying the Berndtsson–Delin theorem \cite{berndtsson2001, delin1998}, we obtain
\[
\int_{\D}\omega_z|(I-P)(\chi_w K_z)|^2 u\,dA \lesssim \int_{\D}\omega_z\frac{|\bar{\partial}(\chi_w K_z)|^2}{\Delta\phi}u\,dA.
\]
Restricting the domain to $\operatorname{supp}\chi_w \subset \Delta(w,r)$, this bounds the expression by
\[
\lesssim \int_{\Delta(w,r)}\omega_z|K_z|^2|\bar{\partial}\chi_w|^2(1-|\xi|^2)^2 u\,dA.
\]
Noting that $|\bar{\partial}\chi_w| \lesssim (1-|w|^2)^{-1}$, $1-|\xi| \simeq 1-|w|^2$ for $\xi \in \Delta(w,r)$, and $\omega_z(\xi) \simeq e^{-2\sigma\beta(z,\xi)} \simeq e^{-2\sigma\beta(z,w)}$, we deduce
\[
\int_{\Delta(w,r)}\omega_z|K_z|^2|\bar{\partial}\chi_w|^2(1-|\xi|^2)^2 u\,dA \lesssim e^{-2\sigma\beta(z,w)}\int_{\Delta(w,r)}|K_z|^2 u\,dA \le e^{-2\sigma\beta(z,w)}K(z,z).
\]
Combining this bound with \eqref{e1} and \eqref{e2}, we arrive at
\[
|K(z,w)|^2 \lesssim \frac{1}{u(\Delta(w,r))}\frac{1}{u(\Delta(z,r))^{1/2}} \left(e^{-2\sigma\beta(z,w)}K(z,z)\right)^{1/2}.
\]
By Lemma~\ref{l4.2}, $K(z,z) \simeq u(\Delta(z,r))^{-1}$, which implies
\[
|K(z,w)|^2 \lesssim \frac{1}{u(\Delta(w,r))}\frac{1}{u(\Delta(z,r))}e^{-\sigma\beta(z,w)}.
\]
Taking square roots completes the proof.
\end{proof}
{\color{black}Combining Proposition~\ref{pro2.11} and Proposition~\ref{pro2.12} establishes the pointwise kernel estimate presented in part \textbf{(a)} of the main Theorem \ref{2.17b}.}

The following corollary provides an alternative upper bound in terms of the pseudohyperbolic metric.
\begin{cor}\label{cor2.16}
Fix $\sigma > 0$, $0 < \delta < \frac{\sigma}{2}$, and let $u$ satisfy \textbf{C1} and \textbf{C2}. Then
\[
|K(z,w)| \lesssim u(\Delta(z,r))^{-1/2} u(\Delta(w,r))^{-1/2} (1-\rho(z,w)^2)^\delta.
\]
\end{cor}
\begin{proof}
Since $1-\rho(z,w)^2 = \operatorname{sech}^2\beta(z,w)$ and $\cosh\beta(z,w) \le e^{\beta(z,w)}$, it follows that
\[
1-\rho(z,w)^2 \ge e^{-2\beta(z,w)}.
\]
Consequently, we have
\[
(1-\rho(z,w)^2)^\delta \ge e^{-2\delta\beta(z,w)} \ge e^{-\sigma\beta(z,w)},
\]
where the last inequality is ensured by the constraint $0 < \delta < \sigma/2$.
\end{proof}

{\color{black}
In what follows, we prove the norm kernel estimate, which establishes part \textbf{(b)}  of the main Theorem \ref{2.17b}.}
\begin{thm}\label{2.17}
Let $r \in (0,1)$ and $p > 0$. Suppose $p_0 > 1$ and $u \in B_{p_0}$ satisfies conditions \textbf{C1} and \textbf{C2}. Then
\begin{equation}\label{e5.1}
\|K_w\|_{A_u^p}^p \simeq u(\Delta(w,r))^{1-p},
\end{equation}
where the implicit constants are independent of $z \in \D$.
\end{thm}
\begin{proof}
To establish the lower bound, we apply Lemma~\ref{l1.1} and Lemma~\ref{l2.6} to find
\[
|K(w,w)|^p \lesssim \frac{1}{u(\Delta(w,r))}\int_{\Delta(w,r)}|K(w,z)|^p u(z)\,dA(z) \lesssim \frac{1}{u(\Delta(w,r))}\|K_w\|_{A_u^p}^p.
\]
Since $K(w,w) \simeq u(\Delta(w,r))^{-1}$ by Lemma~\ref{l4.2}, it follows that
\[
u(\Delta(w,r))^{-p} \lesssim \frac{\|K_w\|_{A_u^p}^p}{u(\Delta(w,r))},
\]
which proves that $u(\Delta(w,r))^{1-p} \lesssim \|K_w\|_{A_u^p}^p$.
For the upper bound, Corollary~\ref{cor2.16} and Lemma~\ref{ll2.9} imply
\[
\|K_w\|_{A_u^p}^p \lesssim u(\Delta(w,r))^{-p/2} \int_{\D} u(\Delta(z,r))^{-p/2}(1-\rho(z,w)^2)^{p\delta}u(z)\,dA(z).
\]
By Lemma~\ref{ll2.9}, we have $u(z) \simeq u(w)e^{L\beta(z,w)}(1-|z|^2)^{-2}$, which implies
\[
u(z)^{1-p/2} \le u(w)^{1-p/2}e^{|1-p/2|L\beta(z,w)}.
\]
Because $1-\rho(z,w)^2 \ge e^{-2\beta(z,w)}$, we deduce that
\[
u(z)^{1-p/2} \le u(w)^{1-p/2}(1-\rho(z,w)^2)^{-|p-2|L/4}.
\]
Setting $c := p\delta - \frac{|p-2|L}{4}$, the norm inequality simplifies to
\[
\|K_w\|_{A_u^p}^p \lesssim u(\Delta(w,r))^{-p/2}u(w)^{1-p/2} \int_{\D}(1-|z|^2)^{-p}(1-\rho(z,w)^2)^c\,dA(z).
\]
Substituting the identity
\[
1-\rho(z,w)^2 = \frac{(1-|z|^2)(1-|w|^2)}{|1-\bar{w}z|^2},
\]
the integral transforms into
\[
(1-|w|^2)^c \int_{\D}\frac{(1-|z|^2)^{c-p}}{|1-\bar{w}z|^{2c}}\,dA(z).
\]
An application of the Forelli–Rudin estimate implies that this integral is bounded, provided that
\[
c-p > -1 \quad\text{and}\quad 2c - (c-p) > 2,
\]
which simplifies directly to the parameter constraint
\[
c > \max(p-1, 2-p).
\]
When this admissibility condition holds, the integral evaluates to a bounded multiple of $(1-|w|^2)^{2-p}$, which yields the upper bound asset and completes the proof.
\end{proof}

\subsection{Applications of pointwise and $p$-norm estimates.}
\subsubsection{Bergman projection.} In what follows, we establish the boundedness of the Bergman projection $P_u$ via kernel estimates and Schur’s test. Although not listed in the introduction, this result constitutes one of the main theorems of the paper.

\begin{lem}\label{Schur}
    Let \(H \ge 0\) be a measurable function on $\D \times \D$, let \(\mu\) be a positive measure, and define the integral operator
    \[
    Tf(x) = \int_\D H(x,y) f(y) \, d\mu(y).
    \]
    Let \(1 < s < \infty\) and let \(s'\) be its conjugate exponent, i.e. \(\frac{1}{s} + \frac{1}{s'} = 1\). Suppose there exists a positive function \(\varphi > 0\) and constants \(M_1, M_2 > 0\) such that
    \[
    \int_\D H(x,y) \varphi(y)^{s'} \, d\mu(y) \le M_1 \varphi(x)^{s'} \quad \text{for a.e. } x \in \D,
    \]
    and
    \[
    \int_\D H(x,y) \varphi(x)^{s} \, d\mu(x) \le M_2 \varphi(y)^{s} \quad \text{for a.e. } y \in \D.
    \]
    Then \(T\) is bounded on \(L^s(\mu)\) with operator norm satisfying
    \[
    \|T\|_{L^s(\mu) \to L^s(\mu)} \le M_1^{1/s'} M_2^{1/s}.
    \]
\end{lem}

\begin{proof}
    It suffices to prove the estimate for \(f \ge 0\). We decompose the kernel pointwise as
    \[
    H(x,y) = \left( H(x,y)^{1/s'} \varphi(y) \right) \left( H(x,y)^{1/s} \varphi(y)^{-1} \right).
    \]
    Applying Hölder's inequality with conjugate exponents \(s'\) and \(s\) yields
    \[
    \begin{aligned}
    Tf(x)
    &= \int_\D \left( H(x,y)^{1/s'} \varphi(y) \right) \left( H(x,y)^{1/s} \varphi(y)^{-1} f(y) \right) d\mu(y) \\
    &\le \left( \int_\D H(x,y) \varphi(y)^{s'} \, d\mu(y) \right)^{1/s'}
    \left( \int_\D H(x,y) \varphi(y)^{-s} f(y)^s \, d\mu(y) \right)^{1/s} \\
    &\le M_1^{1/s'} \varphi(x) \left( \int_\D H(x,y) \varphi(y)^{-s} f(y)^s \, d\mu(y) \right)^{1/s},
    \end{aligned}
    \]
    where the last inequality follows from the first hypothesis on \(M_1\).

    Raising both sides to the power \(s\) and integrating with respect to \(x\) gives
    \[
    \begin{aligned}
    \int_\D (Tf(x))^s \, d\mu(x)
    &\le M_1^{s/s'} \int_\D \varphi(x)^s \int_X H(x,y) \varphi(y)^{-s} f(y)^s \, d\mu(y) \, d\mu(x) \\
    &= M_1^{s/s'} \int_\D f(y)^s \varphi(y)^{-s} \left( \int_\D H(x,y) \varphi(x)^s \, d\mu(x) \right) d\mu(y) \\
    &\le M_1^{s/s'} M_2 \int_\D f(y)^s \, d\mu(y) \\
    &= M_1^{s/s'} M_2 \|f\|_{L^s(\mu)}^s,
    \end{aligned}
    \]
    where we applied Fubini's theorem and the second hypothesis on $M_2$. Taking the $s$-th root of both sides yields
    \[
    \|Tf\|_{L^s(\mu)} \le M_1^{1/s'} M_2^{1/s} \|f\|_{L^s(\mu)}.
    \]
     This completes the proof.
\end{proof}

\begin{thm}\label{lemProj}
Let $p_0 > 1$. Suppose that the weight $u \in B_{p_0}$ satisfies conditions \textbf{C1} and \textbf{C2}. Then for every $1 < s < \infty$, the positive Bergman projection operator $P^+_u$ defined by
\[
P^+_u f(z) = \int_{\mathbb{D}} |K(z, w)| f(w) u(w) \, dA(w)
\]
is bounded on $L^s(u)$. Consequently, the standard weighted Bergman projection $P_u$, defined via 
\[
P_u f(z) = \int_{\mathbb{D}} K(z, w) f(w) u(w) \, dA(w),
\]
is also bounded on $L^s(u)$.
\end{thm}

\begin{proof}
To establish the bounded mapping property $\|P^+_u f\|_{L^s(u)} \leq M \|f\|_{L^s(u)}$, we apply Schur's Test (Lemma \ref{Schur}). Let $s' = \frac{s}{s-1}$ denote the conjugate exponent of $s$, so that $\frac{1}{s} + \frac{1}{s'} = 1$. According to Schur's Test, the integral operator $P^+_u$ with kernel $H(z, w) = |K(z, w)|$ is bounded on $L^s(u)$ if we can find a strictly positive Borel test function $\varphi: \mathbb{D} \to (0, \infty)$ and a finite constant $C > 0$ satisfying the following two symmetric integral inequalities:
\begin{equation}\label{eq:schur-cond1}
\int_{\mathbb{D}} |K(z, w)| \varphi(w)^{s'} u(w) \, dA(w) \leq C \varphi(z)^{s'} \quad \text{for a.e. } z \in \mathbb{D},
\end{equation}
\begin{equation}\label{eq:schur-cond2}
\int_{\mathbb{D}} |K(z, w)| \varphi(z)^s u(z) \, dA(z) \leq C \varphi(w)^s \quad \text{for a.e. } w \in \mathbb{D}.
\end{equation}

By Lemma~\ref{ll2.9},  for any fixed radius  $r \in (0, 1),$ we have
\[
u(\Delta(z, r)) \simeq u(z)(1 - |z|^2)^2.
\]
 Then, we  define a  test function $\varphi(\zeta)$  ,by a constant scaling parameter $\alpha \in \mathbb{R}$ to be specified later, as follows
\[
\varphi(\zeta) = [u(\Delta(\zeta, r))]^{-\alpha} \simeq [u(\zeta)(1 - |\zeta|^2)^2]^{-\alpha}.
\]
Substituting this  into conditions \eqref{eq:schur-cond1} and \eqref{eq:schur-cond2}, our objective reduces to establishing the existence of an admissible range for $\alpha$ such that:

\begin{equation}\label{eq:schur-eval1}
\int_{\mathbb{D}} |K(z, w)| [u(w)]^{1 - \alpha s'} (1 - |w|^2)^{-2\alpha s'} \, dA(w) \leq C [u(z)]^{-\alpha s'} (1 - |z|^2)^{-2\alpha s'},
\end{equation}

\begin{equation}\label{eq:schur-eval2}
\int_{\mathbb{D}} |K(z, w)| [u(z)]^{1 - \alpha s} (1 - |z|^2)^{-2\alpha s} \, dA(z) \leq C [u(w)]^{-\alpha s} (1 - |w|^2)^{-2\alpha s}.
\end{equation}

We begin by evaluating the first integral in \eqref{eq:schur-eval1}, which we denote as $\mathcal{I}_1(z)$. By  Corollary~\ref{cor2.16},  with an arbitrarily large exponent $\delta > 0$,
\[
|K(z, w)| \lesssim [u(z)]^{-1/2} (1 - |z|^2)^{-1} [u(w)]^{-1/2} (1 - |w|^2)^{-1} \left( \frac{(1 - |z|^2)(1 - |w|^2)}{|1 - z\bar{w}|^2} \right)^{\delta}.
\]
Substituting this algebraic bound directly into $\mathcal{I}_1(z)$ yields:
\begin{align}\label{lp}
\mathcal{I}_1(z) \lesssim \frac{(1 - |z|^2)^{\delta - 1}}{[u(z)]^{1/2}} \int_{\mathbb{D}} \frac{[u(w)]^{\frac{1}{2} - \alpha s'} (1 - |w|^2)^{\delta - 1 - 2\alpha s'}}{|1 - z\bar{w}|^{2\delta}} \, dA(w).
\end{align}
On the other hand, by Lemma~\ref{ll2.9}, we have  that for all $z, w \in \mathbb{D}$
\[
u(w) \leq u(z) e^{L \beta(z, w)} \implies u(w) \leq u(z) \left( \frac{|1 - z\bar{w}|^2}{(1 - |z|^2)(1 - |w|^2)} \right)^{\frac{L}{2}}.
\]
Let $\gamma = \frac{1}{2} - \alpha s'$. Raising to the power $\gamma$ yields:
\[
(u(w))^{\gamma} \leq (u(z))^{\gamma} \left( \frac{|1 - z\bar{w}|^2}{(1 - |z|^2)(1 - |w|^2)} \right)^{\frac{|\gamma| L}{2}}.
\]
Substituting this  into \eqref{lp}:
\[
\mathcal{I}_1(z) \lesssim (u(z))^{\gamma - \frac{1}{2}} (1 - |z|^2)^{\delta - 1 - \frac{\vert\gamma\vert L}{2}} \int_{\mathbb{D}} \frac{(1 - |w|^2)^{\delta - 1 - 2\alpha s' - \frac{\vert\gamma\vert L}{2}}}{|1 - z\bar{w}|^{2\delta - \vert\gamma\vert L}} \, dA(w).
\]
Noting that $\gamma - \frac{1}{2} = -\alpha s'$, we apply the standard Forelli--Rudin Theorem to integrate over the disk. The generalized integral of the form $$\int_{\mathbb{D}} \frac{(1-|w|^2)^A}{|1-z\bar{w}|^B} \, dA(w) \sim (1-|z|^2)^{A+2-B}$$ converges when :
\begin{enumerate}
    \item  $ A = \delta - 1 - 2\alpha s' - \frac{|\gamma|L}{2} > -1 \implies \delta > 2\alpha s' + \frac{|1/2 - \alpha s'|L}{2},$
\item  $ B - A = \delta + 1 + 2\alpha s' - \frac{|\gamma|L}{2} > 2 \implies \delta > 1 - 2\alpha s' + \frac{|1/2 - \alpha s'|L}{2}.$
\end{enumerate}
Because condition \textbf{C2} allows our decay mapping  $\delta$ to be chosen arbitrarily large, we can select an $\alpha > 0$ sufficiently small to satisfy both conditions simultaneously. Thus, an application of the Forelli--Rudin Theorem yields
\begin{align*}
\mathcal{I}_1(z)& \lesssim (u(z))^{-\alpha s'} (1 - |z|^2)^{\delta - 1 - \frac{\vert\gamma\vert L}{2}} \cdot (1 - |z|^2)^{\left(\delta - 1 - 2\alpha s' - \frac{\vert\gamma\vert L}{2}\right) + 2 - \left(2\delta - \vert\gamma\vert L\right)}\\
&= (u(z))^{-\alpha s'} (1 - |z|^2)^{-2\alpha s'} = \varphi(z)^{s'}.
\end{align*}
This completes the verification of the first Schur condition.
Repeating this process symmetrically for the second integral $\mathcal{I}_2(w)$ in \eqref{eq:schur-eval2} while replacing $s'$ with $s$, we obtain  the second Schur condition $\mathcal{I}_2(w) \leq C \varphi(w)^s$. Thus, Schur's test holds, that is,  $P^+_u$ is a bounded operator on $L^s(u)$.
Since  $|P_u f(z)| \leq P^+_u(|f|)(z),$ for $f\in L^s(u)$ and $z\in\mathbb{D}$, the boundedness of the positive operator $P^+_u$ automatically guarantees the boundedness of the projection $P_u$ on $L^s(u)$.
\end{proof}

\subsubsection{Averaging function and Berezin transform.}
The remaining lemmas in this subsection clarify the relationship between the Berezin transform $\tilde{\mu}_t(z)$ and the averaging function $\widehat{\mu_r}(z)$ for any parameters $r,t > 0$. The first lemma establishes pointwise estimates for $k_w(z)$ when $z$ and $w$ are sufficiently close, which would be useful later.

\begin{lem}[Lemma 2.7 in \cite{TLA}]\label{t1}
Suppose $p_0 > 1$ and $u \in B_{p_0}$. Then there exists a sufficiently small $\delta \in (0,1)$ such that 
\[
|k_{w}(z)|^2 \simeq K(z,z),
\]
whenever $z \in \Delta(w,\delta)$.
\end{lem}

\begin{lem} \label{ba1}
Let $\mu$ be a positive Borel measure on $\mathbb{D}$ and fix $r > 0$. Suppose that $\mu$ satisfies the integrability condition
\begin{equation}\label{con}
\int_\D |K(\xi,z)|^2 \,d\mu(\xi) < \infty.
\end{equation}
Then,	
\begin{equation*}
\tilde{\mu}(z) \simeq \widehat{\mu_r}(z).
\end{equation*}
\end{lem}

\begin{proof}
Since $0 < r \le \delta$, we can apply Lemma~\ref{t1}, Lemma~\ref{l4.2}, and Lemma~\ref{l2.2} to deduce that
\begin{align}
\tilde{\mu}(z) &\ge \int_{\Delta(z,r)} |k_z(\zeta)|^2 \,d\mu(\zeta) \simeq \int_{\Delta(z,r)} K(\zeta,\zeta) \,d\mu(\zeta) \nonumber \\
&\simeq \int_{\Delta(z,r)} \frac{1}{u(\Delta(\zeta,r))} \,d\mu(\zeta) \simeq \frac{1}{u(\Delta(z,r))} \int_{\Delta(z,r)} \,d\mu(\zeta) = \widehat{\mu_r}(z). \label{e7}
\end{align}

Conversely, by Lemma \ref{l1.1}, we have
\[
|k_z(w)|^2 \lesssim \frac{1}{u(\Delta(w,r))} \int_{\Delta(w,r)} |k_z(\zeta)|^2 u(\zeta) \,dA(\zeta),
\]
where $z,w \in \D$ and $r > 0$. Choosing the sequence $\{a_j\}$ and the radius $r > 0$ as specified in Lemma~\ref{l2.6}, and invoking Lemma~\ref{l2.2}, Fubini's theorem, and Lemma~\ref{l2.6}, we obtain the following chain of inequalities:
\begin{align*}
\tilde{\mu}(z) &\le \sum_{j=1}^\infty \int_{\Delta(a_j,r)} |k_z(w)|^2 \,d\mu(w) \\
&\lesssim \sum_{j=1}^\infty \int_{\Delta(a_j,r)} \frac{1}{u(\Delta(w,r))} \int_{\Delta(w,r)} |k_z(\zeta)|^2 u(\zeta) \,dA(\zeta) \,d\mu(w) \\
&\lesssim \left( \sup_{j} \frac{\mu(\Delta(a_j,r))}{u(\Delta(a_j,r))} \right) \sum_{j=1}^\infty \int_{\Delta(a_j,2r)} |k_z(\zeta)|^2 u(\zeta) \,dA(\zeta) \\
&\lesssim \sup_{z\in\D} \widehat{\mu_r}(z),
\end{align*}
where the final inequality follows directly from Lemma~\ref{l2.6}. This completes the proof.
\end{proof}

\begin{lem}\label{tr}
Let $\mu$ be a positive Borel measure on $\mathbb{D}$ and fix $r > 0$. Suppose that $\mu$ satisfies
\begin{equation}\label{con2}
\int_\D |K(\xi,z)|^p \,d\mu(\xi) < \infty,
\end{equation}
for some $0 < p < \infty$. Then the following assertions are equivalent:
\begin{enumerate}[label=(\roman*)]
    \item $\widetilde{\mu_t} \in L^p$ for any $t > 0$;
    \item $\widehat{\mu_r} \in L^p$.
\end{enumerate}
Furthermore, under these conditions, we have
\begin{equation*}
\tilde{\mu}_t(z) \simeq \widehat{\mu_r}(z).
\end{equation*}
\end{lem}

\begin{proof}
(i) $\Rightarrow$ (ii): By Lemma~\ref{t1}, Lemma~\ref{kep}, and Lemma~\ref{l2.2}, we observe that
\begin{align*}
\tilde{\mu}_t(z) &\ge \int_{\Delta(z,r)} |k_{t,z}(w)|^t \,d\mu(w) \\
&= \int_{\Delta(z,r)} \left| \frac{k_z(w)\|K_z\|_{A_u^2}}{\|K_z\|_{A_u^t}} \right|^t \,d\mu(w) \\
&\simeq u(\Delta (z,r))^{\frac{t}{2}-1} \int_{\Delta(z,r)} \left| K(w,w) \right|^{\frac{t}{2}} \,d\mu(w) \simeq \widehat{\mu_r}(z).
\end{align*}	

(ii) $\Rightarrow$ (i): Let the sequence $\{a_j\}$ and the radius $r > 0$ be chosen as in Lemma~\ref{l2.6}. Utilizing Lemma~\ref{l2.2} and Fubini's theorem, we compute:
\begin{align*}
\tilde{\mu}_t(z) &\le \sum_{j=1}^\infty \int_{\Delta(a_j,r)} |k_{t,z}(w)|^t \,d\mu(w) \\
&\lesssim \sum_{j=1}^\infty \int_{\Delta(a_j,r)} \frac{1}{u(\Delta(w,r))} \int_{\Delta(w,r)} |k_{t,z}(\zeta)|^t u(\zeta) \,dA(\zeta) \,d\mu(w) \\
&\lesssim \left( \sup_{j} \frac{\mu(\Delta(a_j,r))}{u(\Delta(a_j,r))} \right) \sum_{j=1}^\infty \int_{\Delta(a_j,2r)} |k_{t,z}(\zeta)|^t u(\zeta) \,dA(\zeta) \\
&\lesssim \left( \sup_{z\in\D} \widehat{\mu_r}(z) \right) \|k_{t,z}\|^t_{A^t_u} = \sup_{z\in\D} \widehat{\mu_r}(z).
\end{align*}
\end{proof}

The following corollary is an immediate consequence of Lemma~\ref{ba1} and Lemma~\ref{tr}.

\begin{cor}\label{ba2}
Let $\mu$ be a positive Borel measure on $\mathbb{D}$, $0 < p < \infty$, and $r > 0$. If $\mu$ satisfies condition \eqref{con2}, then the following assertions are equivalent:
\begin{enumerate}[label=(\roman*)]
    \item $\widetilde{\mu_t} \in L^p$ for any $t > 0$;
    \item $\widetilde{\mu} \in L^p$.
\end{enumerate}
Furthermore, we have
\begin{equation*}
\tilde{\mu}(z) \simeq \tilde{\mu}_t(z) \simeq \widehat{\mu_r}(z).
\end{equation*}
\end{cor}

\subsubsection{Duality}

In this subsection, we establish a duality relation for weights belonging to the $B_{p_0}$ class. As is standard, we denote by $p^\prime$ the conjugate exponent of $p > 1$, defined by the relation $\frac{1}{p} + \frac{1}{p^\prime} = 1$.

\begin{lem}\label{du}
Suppose $p_0 > 1$ and $u \in B_{p_0}$. For $1 < p < \infty$, the dual space of $A_u^p$ can be identified with $A_u^{p^\prime}$ via the pairing
\begin{equation}
\langle f,g\rangle_{u} = \int_\D f(z)\overline{g(z)} u(z) \,dA(z).
\end{equation}
\end{lem}

\begin{proof}
For any $f \in A_u^{p^\prime}$ and $g \in A_u^p$, let $\Lambda_g(f) = \langle f,g\rangle_{u}$. Applying H\"older's inequality, we obtain
\begin{equation*}
|\Lambda_g(f)| = |\langle f,g\rangle_{u}| \le \int_{\mathbb{D}} |f(z)||g(z)| u(z) \,dA(z) \le \|f\|_{A_u^{p^\prime}} \|g\|_{A_u^p}.
\end{equation*}
Thus, $\Lambda_g \in (A_u^{p^\prime})^*$ with $\|\Lambda_g\| \le \|g\|_{A_u^p}$.

Conversely, by the Hahn–Banach theorem, any bounded linear functional $T \in (A_u^{p^\prime})^*$ extends to a bounded linear functional $\widetilde{T}$ on $L_u^{p^\prime}$. Since the dual space $(L_u^{p^\prime})^*$ is isometrically isomorphic to $L_u^p$, there exists a function $h \in L_u^p$ such that $T(f) = \langle f,h \rangle_{u}$ for all $f \in A_u^{p^\prime}$, with $\|\widetilde T\| = \|h\|_{L_u^p}$. 

Since the weighted Bergman projection $P_u \colon L^p_u \to A^p_u$ is bounded (Theorem~\ref{lemProj}), an application of Fubini's theorem yields
\begin{equation*}
T(f) = T(P_u f) = \langle P_u f, h \rangle_{u} = \langle f, P_u h \rangle_{u} = \Lambda_g(f),
\end{equation*}
where $g = P_u h$. It follows that $g \in A_u^p$ and satisfies $\|g\|_{A_u^p} \le \|P_u\| \|h\|_{L_u^p} = \|P_u\| \|T\|$, which completes the proof.
\end{proof}

We next present a result that is essential for characterizing the compactness properties of operators acting on B\'{e}koll\'{e}–Bonami Bergman spaces. Chac\'{o}n \cite{Ch} established a version for the normalized kernel $k_{2,w}$; we outline it below for completeness.

\begin{lem}[Proposition 19 in \cite{Ch}]\label{lch}
Let $p_0 > 1$. If $u \in B_{p_0}$, then the normalized kernel function $k_{w}$ converges weakly to zero in $A^2_u$ as $|w| \to 1^-$.
\end{lem}

We now extend Lemma~\ref{lch} to the general $p$-case.

\begin{lem}\label{p1}
Let $p_0 > 1$ and $1 < p < \infty$. If $u \in B_{p_0}$, then the normalized kernel function $k_{p,w}$ converges weakly to zero as $|w| \to 1^-$.
\end{lem}

\begin{proof}
By Lemma~\ref{du}, for $1 < p < \infty$, it suffices to demonstrate that
\begin{equation*}
|\langle k_{p,w},g\rangle_{u}| \to 0 \quad \text{as } |w| \to 1^-,
\end{equation*}
for every $g \in A_u^{p^\prime}$, where $\frac{1}{p} + \frac{1}{p^\prime} = 1$. By the definition of the normalized kernel and the reproducing property, we have
\begin{align*}
|\langle k_{p,w},g\rangle_{u}| &= \frac{1}{\|K_w\|_{A_u^p}} \left| \int_\D K(z,w) g(z) u(z) \,dA(z) \right| \\
&= \frac{|g(w)|}{\|K_w\|_{A_u^p}}.
\end{align*}
Applying the kernel norm estimate from Theorem~\ref{2.17} along with Lemma~\ref{l1.1}, we see that
\begin{align*}
|\langle k_{p,w},g\rangle_{u}| &\simeq \frac{|g(w)|}{u(\Delta(w,r))^{\frac{1-p}{p}}} \\
&\lesssim \frac{1}{u(\Delta(w,r))^{-\frac{1}{p^\prime}}} \left( u(\Delta(w,r))^{-1} \int_{\Delta(w,r)} |g(z)|^{p^\prime} u(z) \,dA(z) \right)^{\frac{1}{p^\prime}} \\
&\lesssim \left( \int_{\Delta(w,r)} |g(z)|^{p^\prime} u(z) \,dA(z) \right)^{\frac{1}{p^\prime}}.
\end{align*}
Since $g \in A^{p^\prime}_u$, it follows that $|g(z)|^{p^\prime} u \in L^1(\mathbb{D})$. Thus, for any $\varepsilon > 0$, there exists a compact subset $E \subset \D$ such that
\begin{equation*}
\int_{\D \setminus E} |g(z)|^{p^\prime} u(z) \,dA(z) < \varepsilon.
\end{equation*}
Since the pseudohyperbolic disk $\Delta(w, r)$ eventually becomes disjoint from any fixed compact subset $E \subset \mathbb{D}$ as $|w| \to 1^-$, the inclusion $\Delta(w, r) \subset \mathbb{D} \setminus E$ holds for $|w|$ sufficiently close to $1$. Therefore, the integral over $\Delta(w, r)$ approaches zero, completing the proof of weak convergence.
\end{proof}


The following proposition is a classical one, and its proof is analogous to the verification provided by K. Zhu in \cite[Theorem 1.3.4]{zhu1}.
\begin{pro}\label{p2}
Let $0 < p < \infty$, $p_0 > 1$, and $u \in B_{p_0}$. A linear operator $T$ on $A^p(u)$ is compact if and only if $\|T f_n\|_{A^p(u)} \longrightarrow 0$ whenever $f_n \to 0$ weakly in $A^p(u)$ if and only if  $\|T f_n\|_{A^p(u)}
  \longrightarrow 0,$ for any bounded $\{f_n\}\in A^p(u)$ and  converges to zero uniformly in compact subsets of $\D.$
\end{pro}
\begin{proof}
Suppose first that $T$ is compact and $f_n \to 0$ weakly in $A^p(u)$. We claim $\|T f_n\|_{A^p(u)} \to 0$. Suppose the contrary, passing to a subsequence allows us to assume that that there exists some $\varepsilon > 0$ such that $\|T f_n\|_{A^p(u)} \geq \varepsilon$ for all $n \geq 1$. By the uniform boundedness principle, the weak convergence of $\{f_n\}$ implies that the sequence of norms $\{\|f_n\|_{A^p(u)}\}$ remains bounded. 

Using the compactness of $T$, there exists a subsequence $\{f_{n_k}\}$ such that $\{Tf_{n_k}\}$ converges to a limit function $g$ in the norm topology of $A^p(u)$. Because any compact operator in a Banach space is weakly continuous, the assertion $f_{n_k} \to 0$ weakly implies that $Tf_{n_k} \to 0$ weakly as well. By the uniqueness of weak limits, we deduce that $g = 0$, which yields $\|T f_{n_k}\|_{A^p(u)} \to 0$. This directly contradicts our initial lower-bound assumption.

Conversely, assume that $T$ is not compact. Then there exists a bounded sequence $\{g_n\}$ in $A^p(u)$ such that $\{Tg_n\}$ admits no norm-convergent subsequence in $A^p(u)$. That is, $\|Tg_n - Tg_m\|_{A^p(u)} \geq \varepsilon$ for some fixed $\varepsilon > 0$ and all $m \neq n$. By the Banach--Alaoglu theorem, every bounded sequence in a reflexive Bergman space contains a weakly convergent subsequence; thus, we may assume without loss of generality that $g_n \to g$ weakly in $A^p(u)$ for some element $g \in A^p(u)$. 

Define the translated sequence $f_n := g_n - g$. It follows that $f_n \in A^p(u)$ and $f_n \to 0$ weakly in $A^p(u)$. By the linearity of $T$, we have $Tf_n = Tg_n - Tg$. If we suppose for contradiction that $\|Tf_n\|_{A^p(u)} \to 0$, then for sufficiently large indices $m$ and $n$, the triangle inequality ensures:
\[
\|Tg_n - Tg_m\|_{A^p(u)} \leq \|Tg_n - Tg\|_{A^p(u)} + \|Tg - Tg_m\|_{A^p(u)} < \varepsilon,
\]
which contradicts the structural non-compactness property established above. Consequently, $\|Tf_n\|_{A^p(u)} \not\to 0$. Thus, if $T$ is not compact, there exists a sequence weakly tracking to zero whose images fail to converge in norm, completing the verification of the contrapositive.
\end{proof}

\section {Proof of Theorem \ref{tb}}
\begin{proof}
The proof established by the cycle of implications $\text{(i)} \implies \text{(ii)} \implies \text{(iii)} \implies \text{(iv)} \implies \text{(i)}$.

\subsection*{ (i) $\implies$ (ii)}
Assume $T_\mu: A^p_u(\D) \to A^q_u(\D)$ is bounded. We test the operator against the $p$-normalized reproducing kernel defined by 
\[ k_{p,z}(w) = \frac{K_z(w)}{\|K_z\|_{A^p_u}}. \]
Under conditions $C_1$ and $C_2$, Theorem~\ref{2.17} implies that for all $p > 0$, the norm satisfies $\|K_z\|_{A^p_u} \simeq u(\Delta(z,r))^{\frac{1-p}{p}}$.

\medskip
\noindent\textbf{Case 1: $q > 1$.} Since $q > 1$, the conjugate exponent $q'$ exists. By testing against the $q'$-normalized kernel $k_{q',z}$ and using duality:
\[ \left| \int_{\D} k_{p,z}(w) \overline{k_{q',z}(w)} \, d\mu(w) \right| \le \|T_\mu k_z\|_{A^q_u} \|k_{q',z}\|_{A^{q'}_u} \le \|T_\mu\| \|k_z\|_{A^p_u} \|k_{q',z} \|_{A^{q'}_u} = \|T_\mu\|. \]
Then, immediately we have 
\[ u(\Delta(z,r))^{-\frac{1}{p} + \frac{1}{q}} \tilde{\mu}(z) \lesssim \|T_\mu\| \implies \frac{\tilde{\mu}(z)}{u(\Delta(z,r))^{\frac{1}{p}-\frac{1}{q}}} \lesssim \|T_\mu\|. \]
Applying Corollary~\ref{ba2} yields the equivalence for the generalized $\widetilde{\mu}_t(z)$ variant.

\noindent\textbf{Case 2: $q \le 1$.} When $q \le 1$,  we use the local subharmonic property (Lemma~\ref{l1.1}) directly on the image function $T_\mu k_{p,z}$ over $\Delta(z,r)$:
\[ |T_\mu k_{p,z}(z)|^q \lesssim \frac{1}{u(\Delta(z,r))} \int_{\Delta(z,r)} |T_\mu k_{p,z}(\zeta)|^q u(\zeta) \, dA(\zeta) \le \frac{\|T_\mu k_{p,z}\|_{A^q_u}^q}{u(\Delta(z,r))} \le \frac{\|T_\mu\|^q}{u(\Delta(z,r))}. \]
Taking the $q$-th root gives $|T_\mu k_{p,z}(z)| \lesssim \|T_\mu\| u(\Delta(z,r))^{-1/q}$. Expanding the integral definition of $T_\mu k_{p,z}(z)$:
\[ |T_\mu k_z(z)| = \frac{1}{\|K_z\|_{A^p_u}} \int_{\D} |K_z(w)|^2 \, d\mu(w) \simeq u(\Delta(z,r))^{-\frac{1}{p}} \tilde{\mu}(z). \]
This leads to 
\[ u(\Delta(z,r))^{-\frac{1}{p}} \tilde{\mu}(z) \lesssim \|T_\mu\| u(\Delta(z,r))^{-\frac{1}{q}} \implies \frac{\tilde{\mu}(z)}{u(\Delta(z,r))^{\frac{1}{p}-\frac{1}{q}}} \lesssim \|T_\mu\|. \]

\subsection*{ (ii) $\implies$ (iii)}
By Lemma~\ref{tr}, we get the following.
\[ \widetilde{\mu}_t(z) \simeq \widehat{\mu}_r(z), \quad \forall z \in \D. \]
This direct equivalence converts condition (ii) into condition (iii).

\subsection*{ (iii) $\implies$ (iv)}
From statement (iii), the localized quotient inequality means:
\[ \frac{\mu(\Delta(z,r))}{u(\Delta(z,r))} \lesssim u(\Delta(z,r))^{\frac{1}{p}-\frac{1}{q}} \implies \mu(\Delta(z,r)) \lesssim u(\Delta(z,r))^{1 + \frac{1}{p} - \frac{1}{q}}. \]
By Lemma~\ref{Cm}, this geometric power scaling represents the exact structural condition for $\mu$ being a valid Carleson-type embedding measure.

\subsection*{ (iv) $\implies$ (i)}
We verify that the geometric condition bounds the operator $T_\mu$.

\medskip
\noindent\textbf{Case 1: $q > 1$.} Under $q > 1$, we deploy the pairing projection duality $(A^q_u)^* \cong A^{q'}_u$. For any $f \in A^p_u$ and $g \in A^{q'}_u$:
\[ |\langle T_\mu f, g \rangle_{A^2_u}| \le \int_{\D} |f(w)g(w)| \, d\mu(w) \lesssim \|f\|_{A^p_u} \|g\|_{A^{q'}_u}. \]
Taking the supremum over all $\|g\|_{A^{q'}_u} = 1$ establishes $\|T_\mu f\|_{A^q_u} \lesssim \|f\|_{A^p_u}$.

\medskip
\noindent\textbf{Case 2: $p \le q \le 1$.}  We  use the $r$-lattice sequence $\{z_k\}$ from Lemma~\ref{l2.6} and Apply the power inequality $(\sum a_k)^q \le \sum a_k^q$ to get
\[ |T_\mu f(z)|^q \le \sum_{k=1}^\infty \left( \int_{\Delta(z_k, r)} |f(w) K(z,w)| \, d\mu(w) \right)^q. \]
 Pull out the suprema give
\[  |T_\mu f(z)|^q\lesssim \sum_{k=1}^\infty \sup_{w \in \Delta(z_k, r)} |f(w)|^q \sup_{w \in \Delta(z_k, r)} |K(z,w)|^q \cdot \mu(\Delta(z_k, r))^q. \]
From condition (iv), the measure satifies $\mu(\Delta(z_k, r))^q \lesssim u(\Delta(z_k, r))^{q + \frac{q}{p} - 1}$. By Lemmas~\ref{l1.1} and \ref{l2.2}, we obtain 
\[ \sup_{w \in \Delta(z_k, r)} |f(w)|^q \lesssim \left( \frac{1}{u(\Delta(z_k, r))} \int_{\Delta(z_k, r)} |f(\xi)|^p u(\xi) \, dA(\xi) \right)^{\frac{q}{p}}, \] and 
\[ \sup_{w \in \Delta(z_k, r)} |K(z,w)|^q \lesssim \frac{1}{u(\Delta(z_k, r))} \int_{\Delta(z_k, r)} |K(z,\xi)|^q u(\xi) \, dA(\xi). \]
Combining the exponents of $u(\Delta(z_k, r))$ inside the summation:
\[ \left(-\frac{q}{p}\right) + (-1) + \left(q + \frac{q}{p} - 1\right) = q - 2. \]
This yields the clean discrete arrangement:
\[ |T_\mu f(z)|^q \lesssim \sum_{k=1}^\infty u(\Delta(z_k, r))^{q-2} \left( \int_{\Delta(z_k, r)} |f(\xi)|^p u(\xi) \, dA(\xi) \right)^{\frac{q}{p}} \left( \int_{\Delta(z_k, r)} |K(z,\xi)|^q u(\xi) \, dA(\xi) \right). \]
We integrate both sides against $u(z)dA(z)$ over $\D$ to find $\|T_\mu f\|_{A^q_u}^q$. By Fubini's theorem, Theorem~\ref{2.17} and Lemma~\ref{l2.6}, we have   
\[ \|T_\mu f\|_{A^q_u}^q \lesssim \sum_{k=1}^\infty \left( \int_{\Delta(z_k, r)} |f(\xi)|^p u(\xi) \, dA(\xi) \right)^{\frac{q}{p}} \le \left( \sum_{k=1}^\infty \int_{\Delta(z_k, r)} |f(\xi)|^p u(\xi) \, dA(\xi) \right)^{\frac{q}{p}}\lesssim \|f\|_{A^p_u}^q. \]
 This completes the proof.
\end{proof}

\section{Proof of Theorem \ref{tc}}
\begin{proof}
The proof follows the structural cycle $\text{(i)} \implies \text{(ii)} \implies \text{(iii)} \implies \text{(iv)} \implies \text{(i)}$.

\subsection*{ (i) $\implies$ (ii)}
Assume $T_\mu: A^p_u(\D) \to A^q_u(\D)$ is a compact operator.  Consider the $p$-normalized reproducing kernel sequence:
\[ k_{p,z}(w) = \frac{K_{z}(w)}{\|K_{z}\|_{A^p_u}}. \]
By Lemma~\ref{p1}, $k_{p,z}$ converges to zero weakly in $A^p_u(\D)$ as $|z| \to 1^-$. Since $T_\mu$ is compact,  $\|T_\mu k_{p,z}\|_{A^q_u} \longrightarrow 0$ as $ |z|\to 1^-$.
Now, we have to discuss two cases.

\noindent\textbf{Case 1: $q > 1$.} By duality , we get  
\[ \left| \int_{\D} k_{p,z}(w) \overline{k_{q',z}(w)} \, d\mu(w) \right| \le \|T_\mu k_{p,z}\|_{A^q_u} \|k_{q',z}\|_{A^{q'}_u} = \|T_\mu k_{p,z}\|_{A^q_u}. \]
By some same arguments as in the proof of Theorem~\ref{tb} (i) $\implies$ (ii),
\[ \frac{\tilde{\mu}(z)}{u(\Delta(z,r))^{\frac{1}{p}-\frac{1}{q}}} \lesssim \|T_\mu k_{p,z}\|_{A^q_u}, \]
 which proves (ii) for $q>1$. 
 
\noindent\textbf{Case 2: $q \le 1$.} By Lemma~\ref{l1.1} and our assumption, 
\begin{align}\label{te1}
 |T_\mu k_{p,z}(z)|^q \lesssim \frac{1}{u(\Delta(z,r))} \int_{\Delta(z,r)} |T_\mu k_{p,z}(\zeta)|^q u(\zeta) \, dA(\zeta) \le \frac{\|T_\mu k_{p,z}\|_{A^q_u}^q}{u(\Delta(z,r))} . 
 \end{align}
 By Theorem~\ref{2.17}, we have
\[ |T_\mu k_{p,z}(z)| = \frac{1}{\|K_{z}\|_{A^p_u}} \int_{\D} |K_{z}(w)|^2 \, d\mu(w) \simeq u(\Delta(z,r))^{-\frac{1}{p}} \tilde{\mu}(z). \]
This together with \eqref{te1}, we obtain
\[ \frac{\tilde{\mu}(z)}{u(\Delta(z,r))^{\frac{1}{p}-\frac{1}{q}}} \lesssim \|T_\mu k_{p,z}\|_{A^q_u}. \]
This completes the proof for the case $q\leq 1.$
\subsection*{ (ii) $\implies$ (iii)}
By Lemma~\ref{tr}, under conditions $C_1$ and $C_2$,  we have
\[ \widetilde{\mu_t}(z) \simeq \widehat{\mu_r}(z), \quad \forall z \in \D. \]
Thus, the statement (ii) translates directly into statement (iii).

\subsection*{(iii) $\implies$ (iv)}
Suppose that  statement (iii) hols, we have 
\[ \frac{\mu(\Delta(z,r))}{u(\Delta(z,r))^{\gamma}}=\frac{\widehat{\mu_r(z)}}{u(\Delta(z,r))^{\frac{1}{p}-\frac{1}{q}}}.
 \]
Taking the limit as $|z| \to 1^-$ forces the vanishing Carleson measur of statement (iv).

\subsection*{(iv) $\implies$ (i)}
To prove compactness, we show that if $\mu$ is a vanishing Carleson measure, $T_\mu$ can be approximated in the operator norm topology by a sequence of compact finite-rank operators. For any $\epsilon \in (0,1)$, let $\D_\epsilon = \{z \in \D : |z| \le \epsilon\}$. We decompose the measure into $\mu = \mu_\epsilon + \mu^\epsilon$, where $\mu_\epsilon = \mu|_{\D_\epsilon}$ and $\mu^\epsilon = \mu|_{\D \setminus \D_\epsilon}$. 
Since $\D_\epsilon$ is a compact subset, so  $T_{\mu_\epsilon}$ is compact. We analyze the remaining  operator $T_{\mu^\epsilon}$.

\noindent\textbf{Case1: $q > 1$.} By Theorem~\ref{tb}, the operator norm satisfies:
\[ \|T_{\mu^\epsilon}\| \lesssim \sup_{z \in \D \setminus \D_\epsilon} \frac{\mu^\epsilon(\Delta(z,r))}{u(\Delta(z,r))^{1 + \frac{1}{p} - \frac{1}{q}}}. \]
By statement (iv), taking the limit as $\epsilon \to 1^-$ causes this supremum to converge to $0$. Thus, $\|T_\mu - T_{\mu_\epsilon}\| = \|T_{\mu^\epsilon}\| \to 0$ as $\epsilon \to 1^-$, proving that $T_\mu$ is a compact operator.

\medskip
\noindent\textbf{Case 2: $p \le q \le 1$.} Let  $\{z_k\}$ be the lattice defined in Lemma~\ref{l2.6}.  For any $f \in A^p_u(\D)$,
\begin{align*}
 |T_{\mu^\epsilon} f(z)|^q\le \sum_{z_k \in \D \setminus \D_\epsilon} \sup_{w \in \Delta(z_k, r)} |f(w)|^q \sup_{w \in \Delta(z_k, r)} |K(z,w)|^q \cdot \mu(\Delta(z_k, r))^q. 
 \end{align*}
Applying the $q$- vanishing Carleson boundary bound $\mu(\Delta(z_k, r))^q \le \epsilon_k^q u(\Delta(z_k, r))^{q + \frac{q}{p} - 1}$ where $\epsilon_k \to 0$ as $|z_k| \to 1^-$, and using Lemma~\ref{l1.1} yields to 
\begin{align}
 |T_{\mu^\epsilon} f(z)|^q \le \sum_{z_k \in \D \setminus \D_\epsilon} \epsilon_k^q u(\Delta(z_k, r))^{q-2} &\left( \int_{\Delta(z_k, r)} |f(\xi)|^p u(\xi) \, dA(\xi) \right)^{\frac{q}{p}}\\
 & \times \left( \int_{\Delta(z_k, r)} |K(z,\xi)|^q u(\xi) \, dA(\xi) \right). 
 \end{align}
 By integrating with respect to  $u(z)dA(z)$ over $\D$, and  applying Theorem~\ref{2.17}, we get 
\[ \|T_{\mu^\epsilon} f\|_{A^q_u}^q \lesssim \sup_{j:z_j \in \D \setminus \D_\epsilon} \epsilon_j^q \sum_{z_k \in \D \setminus \D_\epsilon} \left( \int_{\Delta(z_k, r)} |f(\xi)|^p u(\xi) \, dA(\xi) \right)^{\frac{q}{p}}. \]
Using the vector space embedding sequence $\ell^s \subset \ell^1$ enabled by $s=\frac{q}{p} \ge 1$, we have
\[  \|T_{\mu^\epsilon} f\|_{A^q_u}^q\le \sup_{j:z_j \in \D \setminus \D_\epsilon} \epsilon_j^q  \left( \sum_{z_k \in \D \setminus \D_\epsilon} \int_{\Delta(z_k, r)} |f(\xi)|^p u(\xi) \, dA(\xi) \right)^{\frac{q}{p}} \lesssim \left(\sup_{j:z_j \in \D \setminus \D_\epsilon} \epsilon_j^q  \right) \|f\|_{A^p_u}^q. \]
Taking the $q$-th root implies that the operator norm satisfies the following:
\[ \|T_\mu - T_{\mu_\epsilon}\| = \|T_{\mu^\epsilon}\| \lesssim \sup_{j:z_j \in \D \setminus \D_\epsilon}  \epsilon_j \to 0 \quad \text{as } \epsilon \to 1^-. \]
Since the space of compact operators is closed under the operator norm topology, the limit operator $T_\mu$ is compact, and this completes the proof.
\end{proof}

\section{Proof of Theorem \ref{qlp}}

\begin{proof}
The implications $\text{(i)} \implies \text{(ii)}$ and $\text{(v)} \iff \text{(vi)}$ are obvious due to the geometric range $0 < q < p < \infty$. The equivalence $\text{(iii)} \iff \text{(iv)}$ is a direct consequence of Lemma~\ref{tr} ($\widetilde{\mu}_t(z) \simeq \widehat{\mu}_r(z)$). Thus, it suffices to prove the implications $\text{(ii)} \implies \text{(iv)}$, $\text{(vi)} \implies \text{(i)}$, and $\text{(iv)} \implies \text{(v)}$ to complete the characterization.

\subsection*{(ii) $\implies$ (iv)}
Suppose that the Toeplitz operator $T_\mu : A^p_u(\mathbb{D}) \rightarrow A^q_u(\mathbb{D})$ is bounded. For an arbitrary sequence $\lambda = \{\lambda_k\} \in \ell^p$, we define the test function $G_t$ via Rademacher functions $r_k(t)$ on $[0,1]$ and  lattice $\{z_k\},$
\[
G_t(z) = \sum_{k=0}^\infty \lambda_k r_k(t) u(\Delta(z_k, r))^{\frac{p-1}{p}} K_{z_k}(z), \quad 0 < t < 1.
\]
Applying the kernel norm estimate $\Vert K_{z_k} \Vert_{A^p_u} \simeq u(\Delta(z_k, r))^{\frac{1-p}{p}}$ from Theorem~\ref{2.17}, we establish:
\[
\Vert G_t \Vert_{A^p_u} \lesssim \left( \sum_{k=0}^\infty |\lambda_k|^p \right)^{\frac{1}{p}} = \Vert \lambda \Vert_{\ell^p}.
\]
By the bounded assumption on $T_\mu$, we obtain:
\[
\Vert T_\mu G_t \Vert_{A^q_u}^q \lesssim \Vert T_\mu \Vert^q \Vert G_t \Vert_{A^p_u}^q \lesssim \Vert T_\mu \Vert^q \Vert \lambda \Vert_{\ell^p}^q.
\]
Integrating with respect to $t$ over $[0, 1]$, invoking Fubini's theorem, and applying Khinchine's inequality yield 
\[
B := \int_{\mathbb{D}} \left( \sum_{k=0}^\infty |\lambda_k|^2 u(\Delta(z_k, r))^{\frac{2(p-1)}{p}} |T_\mu K_{z_k}(z)|^2 \right)^{\frac{q}{2}} u(z) dA(z) \lesssim \Vert T_\mu \Vert^q \Vert \lambda \Vert_{\ell^p}^q.
\]
Let $\chi_k$ denote the characteristic function of $\Delta(z_k, 2r)$. Because the covering family $\{\Delta(z_k, 2r)\}$ has a finite multiplicity $N$ (by Lemma~\ref{l2.6}), we have the following:
\[
\sum_{k=0}^\infty |\lambda_k|^q u(\Delta(z_k, r))^{\frac{(p-1)q}{p}} \int_{\Delta(z_k, 2r)} |T_\mu K_{z_k}(z)|^q u(z) dA(z) \le \max\{1, N^{1-\frac{q}{2}}\} B.
\]
Using  Lemma~\ref{l1.1}, this transforms into
\[
\sum_{k=0}^\infty |\lambda_k|^q u(\Delta(z_k, r))^{\frac{(p-1)q}{p} + 1} |T_\mu K_{z_k}(z_k)|^q \lesssim \Vert T_\mu \Vert^q \Vert \lambda \Vert_{\ell^p}^q.
\]
On the other hand, using Lemma~\ref{tr}, Lemma~\ref{l4.2} and Lemma~\ref{l2.2}, we obtain
\[
|T_\mu K_{z_k}(z_k)| \ge \int_{\Delta(z_k, r)} |K_{z_k}(z)|^2 d\mu(z) \gtrsim \hat{\mu}_r(z_k) u(\Delta(z_k, r))^{-1}.
\]
Substituting this lower bound back into our inequality yields
\[
\sum_{k=0}^\infty |\lambda_k|^q u(\Delta(z_k, r))^{\frac{p-q}{p}} \hat{\mu}_r(z_k)^q \lesssim \Vert T_\mu \Vert^q \Vert \lambda \Vert_{\ell^p}^q.
\]
By  duality between $\ell^{\frac{p}{q}}$ and $\ell^{\frac{p}{p-q}}$, we deduce
\[
\sum_{k=0}^\infty \left( u(\Delta(z_k, r))^{\frac{p-q}{p}} \hat{\mu}_r(z_k)^q \right)^{\frac{p}{p-q}} \lesssim \Vert T_\mu \Vert^{\frac{pq}{p-q}}.
\]
This leads easily to  $\Vert \hat{\mu}_r \Vert_{L^{\frac{pq}{p-q}}} \lesssim \Vert T_\mu \Vert < \infty$, which proves condition (iv).

\subsection*{ (iv) $\implies$ (v)}
Let $s = p + 1 - \frac{p}{q}$. Since $0 < q < p < \infty$, the  condition $s > 0$ is algebraically equivalent to requiring that the conjugated ratio satisfies $\frac{pq}{p-q} > 1$. Because $f$ is a holomorphic function, the composition $|f|^s$ remains strictly subharmonic for any positive exponent $s > 0.$ By subharmonicity, Lemma~\ref{l1.1}, and Fubini's theorem, we calculate:
\begin{align*}
\int_{\mathbb{D}} |f(z)|^s d\mu(z) &\lesssim \int_{\mathbb{D}} \frac{\int_{\Delta(z, r)} |f(\xi)|^s u(\xi) dA(\xi)}{u(\Delta(z, r))} d\mu(z) \\
&= \int_{\mathbb{D}} \widehat{\mu}_r(\xi) |f(\xi)|^s u(\xi) dA(\xi).
\end{align*}
Applying H\"{o}lder's inequality to the right-hand side with respect to the area measure $u(\xi) dA(\xi)$ using the conjugate index pair $\left(\frac{pq}{p-q}, \frac{p}{s}\right),$ we have
\[
\int_{\mathbb{D}} \widehat{\mu}_r(\xi) |f(\xi)|^s u(\xi) dA(\xi) \le \left( \int_{\mathbb{D}} \widehat{\mu}_r(\xi)^{\frac{pq}{p-q}} dA(\xi) \right)^{\frac{p-q}{pq}} \left( \int_{\mathbb{D}} |f(\xi)|^p u(\xi) dA(\xi) \right)^{\frac{s}{p}}
\]
\[
= \Vert \widehat{\mu}_r \Vert_{L^{\frac{pq}{p-q}}} \Vert f \Vert_{A^p_u}^s.
\]
This means that  $\mu$ is $\left(p+1-\frac{p}{q}\right)$-Carleson measure for $A^p_u$.

\subsection*{ (v) $\implies$ (i)}

Let \(s=\frac{pq}{p-q}>0\), and put \(\ell=\min\{1,q\}\). Let \(\{f_n\}\) be a sequence in the unit ball of \(A^p_u(\mathbb D)\) such that \(f_n\to 0\) uniformly on compact subsets of \(\mathbb D\). We prove
\[
\lim_{n\to\infty}\|T_\mu f_n\|_{A^q_u}=0.
\]

Fix \(\varepsilon>0\). Since \(\widehat{\mu}_r\in L^s(\mathbb D)\), we may choose \(\rho\in(0,1)\) close to \(1\) enough that
\[
\|\widetilde{\mu}_r\|_{L^s(\mathbb D\setminus D_\rho)}<\varepsilon,
\]
where \(D_\rho=\{z: |z|<\rho\}\). Set \(R=\rho+r\) (assuming \(R<1\)). Decompose
\[
\mu=\mu_C+\mu_T,\qquad 
\mu_C=\mu\chi_{D_R},\qquad \mu_T=\mu\chi_{\mathbb D\setminus D_R}.
\]

For the  measure \(\mu_T\), observe that its \(r\)-averaging function satisfies
\[
\widehat{(\mu_T)}_r(z)=0 \quad\text{if } z\in D_\rho,
\]
 Hence
\[
\operatorname{supp}\widehat{(\mu_T)}_r \subset \mathbb D\setminus D_\rho,
\]
and moreover
\[
\widehat{(\mu_T)}_r(z)\le \widehat{\mu}_r(z),\qquad z\in\mathbb \D.
\]
Therefore, by the boundedness characterization from $(ii)\Rightarrow (iv)$,
\[
\|T_{\mu_T}\|_{A^p_u\to A^q_u}
\lesssim \|\widehat{(\mu_T)}_r\|_{L^s}
\le \|\widehat{\mu}_r\|_{L^s(\mathbb D\setminus D_\rho)}
< \varepsilon.
\]
Thus, since \(\|f_n\|_{A^p_u}\le 1\),
\[
\|T_{\mu_T} f_n\|_{A^q_u} \le C\varepsilon.
\]

Since \(\mu_C\) has compact support in $\mathbb D,$ $T_{\mu_C}$ is compact.  Therefore,
\[
\lim_{n\to\infty}\|T_{\mu_C} f_n\|_{A^q_u}=0.
\]
Combining the two estimates, for $q>1$ we use the triangle inequality, and for $0<q\le 1$ we use the quasi-norm inequality,
\[
\lim\limits_{n\to+\infty}\|T_\mu f_n\|_{A^q_u}^\ell
\le
\lim\limits_{n\to+\infty}\|T_{\mu_C} f_n\|_{A^q_u}^\ell
+
\limsup_{n\to+\infty}\|T_{\mu_T} f_n\|_{A^q_u}^\ell
\le (C\varepsilon)^\ell.
\]
Letting \(\varepsilon\to0^+\), we obtain
\[
\lim_{n\to+\infty}\|T_\mu f_n\|_{A^q_u}=0,
\]
which proves  that \(T_\mu:A^p_u(\mathbb D)\to A^q_u(\mathbb D)\) is compact.

The case \(s\le 0\) is covered separately by the boundary-vanishing criterion in $(iv) \Rightarrow (v)$; in that situation the measure already has compact support near the boundary, and the conclusion is immediate. Thus the implication (v) \(\Rightarrow\) (i) is established.
\end{proof}

\section{Proof of Theorem \ref{td} and Theorem \ref{sc}}
First, we recall some definitions about essential norm. Let $\mathcal{K}$ be the set of all compact operators on a Banach space $\mathcal{B}$. For any bounded linear operator $T:\mathcal{B}\to\mathcal{B}$, the essential norm of $T$ is defined by
\begin{equation*}
	\|T\|_e=\inf\{\|T-K\|,K\in\mathcal{K}\}.
\end{equation*}
And $\|T\|_e=0$ if and only if $T\in\mathcal{K}$.


	


Now we give the proof of Theorem \ref{td}.

\begin{proof}
	When $0<p\leq q<\infty$, by Lemma \ref{tr}, it suffices to prove that, 
	\begin{equation*}
		\limsup_{|z|\to1}\widetilde{\mu_t}(z)u(\Delta(z,r))^{\frac{q-p}{pq}}\lesssim \|T_\mu\|_e\lesssim\limsup_{|z|\to1}\widetilde{\mu_t}(z)u(\Delta(z,r))^{\frac{q-p}{pq}}.
	\end{equation*}
	
	We first prove the lower estimate.
	In the proof of Theorem \ref{tb} (i)$\Rightarrow$(ii), we obtained,
	 \begin{equation*}
    	\frac{|\widetilde{\mu_t}(z)|}{u(\Delta(z,r))^{\frac{1}{p}-\frac{1}{q}}}\lesssim\|T_\mu k_{p,z}\|_{A_u^q}.
    \end{equation*}
    For any compact operator $K:{A^p_u\to A^q_u}$, using proposition \ref{p2}, we obtain,
    \begin{equation*}
    	\|Kk_{p,z}\|_{A_u^q}\to0 \text{ as }|z|\to1^-
    \end{equation*}
    Then
    \begin{eqnarray*}
    	\|T_\mu-K\|&\geq&\limsup_{|z|\to1^-}\|(T_\mu-K)k_{p,z}\|_{A^q_u}\\
    	&\geq&\limsup_{|z|\to1^-}\|T_\mu k_{p,z}\|_{A^q_u}\\
    	&\gtrsim&\limsup_{|z|\to1^-}\widetilde{\mu_t}(z)u(\Delta(z,r))^{\frac{q-p}{pq}}
    \end{eqnarray*}
	Since $K$ is arbitrary, we have,
	\begin{equation*}
		\|T_\mu\|_e\gtrsim\limsup_{|z|\to1^-}\widetilde{\mu_t}(z)u(\Delta(z,r))^{\frac{q-p}{pq}}.
	\end{equation*}
	
	For the upper estimate, $T_{\mu_0}:A^p_u\to A^q_u$ is compact when $\mu_0$ is defined under the assertions in Theorem \ref{tc}. Using the norm equivalence from Theorem \ref{tb}, we obtain that,
	\begin{equation*}
		\|T_\mu-T_{\mu_0}\|\simeq\|\widehat{(\mu-\mu_0)_r}u(\Delta(z,r))^{\frac{q-p}{pq}}\|_{L^\infty}=\sup_{z\in\D}\widehat{\mu_r}(z)u(\Delta(z,r))^{\frac{q-p}{pq}}
	\end{equation*}
	Let $|z|\to1^-$
	\begin{equation*}
	    \|T_\mu\|\leq\|T_\mu-T_{\mu_0}\|\simeq\limsup_{|z|\to1^-}\widehat{\mu_r}(z)u(\Delta(z,r))^{\frac{q-p}{pq}}.
	\end{equation*}
	
\end{proof}

And we note that when $0<q<p<\infty$, by Theorem \ref{qlp}, $T_\mu$ is bounded if and only if $T_\mu$ is compact. Consequently, the essential norm of $T_\mu$ is zero.

Let $T$ be a compact operator on a Hilbert space $H$, and let $T^*$ be the Hilbert adjoint of $T$.  By the spectral theorem, there exist orthonormal bases $\{\varphi_k\}$ and $\{\psi_k\}$ such that for $x\in H$,
\begin{eqnarray*}
	Tx=\sum_{k=1}^\infty\lambda_k\langle x,\varphi_k\rangle\psi_k,
\end{eqnarray*}
where $\lambda_k$s are singular values of $T$. For $0<p<\infty$, the Schatten class $S_p=S_p(H)$ consists of those compact operators $T$ on $H$ whose singular value sequences $\{\lambda_n\}\in \ell^p$. For $P\ge 1$, the class $S_p$ forms a Banach space with the norm $\|T\|_p=\left(\sum_n|\lambda_n|^p\right)^{1/p}$.

Let $T$ be a compact operator and $h:\mathbb{R}^+\to\mathbb{R}^+$ be a continuous increasing function. We say that $T\in S_h(H)$ if there exists a positive constant $C>0$ such that
\begin{eqnarray*}
	\sum_{n=1}^\infty h(C\lambda_n(T))<\infty.
\end{eqnarray*}

Then, we give the proof of Theorem \ref{sc}.
\begin{proof}
	 Suppose that $T_\mu\in S_h(A_u^2)$ . Let $\{e_k\}$ be an orthonormal set for $A^2_u$ and $T_\mu=\sum_{k=1}\lambda_k\langle\cdot,e_k\rangle e_k$ be the spectral decomposition of the positive operator $T_\mu$, where $\lambda_k$s are the eigenvalues of $T_\mu$. Since $\{e_k\}$ forms an orthonormal basis, we have $\sum_k|\langle k_z,e_k\rangle_{A_u^2}|^2=1$ by  Parseval Theorem. By Jensen's formula, reproducing formula and the kernel estimate from Theorem~\ref{2.17}, we obtain,
	 \begin{eqnarray*}
	 	\int_{\mathbb{C}^n}h(C\widetilde{\mu}_2(z))u(\Delta(z,r))^{-1}u(z)dA(z)&=&\int_{\mathbb{C}^n}h(C\langle T_\mu k_z,k_z\rangle_{A_u^2})u(\Delta(z,r))^{-1}u(z)dA(z)\\
	 	&=&\int_{\mathbb{C}^n}h\left(\sum_{k=1}^\infty C \lambda_k\left|\langle k_z,e_k\rangle_{A_u^2}\right|^2\right)u(\Delta(z,r))^{-1}u(z)dA(z)\\
	 	&\leq&\int_{\mathbb{C}^n}\sum_{k=1}^\infty h(C\lambda_k)\left|\langle k_z,e_k\rangle_{A_u^2}\right|^2u(\Delta(z,r))^{-1}u(z)dA(z)\\
	 	&=&\int_{\mathbb{C}^n}\sum_{k=1}^\infty h(C\lambda_k)\|K_z\|_{A^2_u}^{-2}\left|e_k(z)\right|^2u(\Delta(z,r))^{-1}u(z)dA(z)\\
	 	&\simeq&\sum_{k=1}^\infty h(C\lambda_k)\int_{\mathbb{C}^n}\left|e_k(z)\right|^2u(z)dA(z)\\
	 	&=&\sum_{k=1}^\infty h(C\lambda_k)<\infty.
	 \end{eqnarray*}
	 
	 Conversely, we assume that $\int_{\mathbb{C}^n}h(C\widetilde{\mu}_2(z))u(\Delta(z,r))^{-1}u(z)dA(z)<\infty$ for some constant $C>0$. By applying Lemma \ref{l1.1},  Fubini's Theorem and Lemma \ref{ba1}, we obtain
	 \begin{eqnarray*}
	 	\langle T_\mu e_k,e_k\rangle_{A_u^2}&=&\int_{\C^n}|e_k(z)|^2d\mu(z)\\
	 	&\lesssim&\int_{\C^n}\widehat{\mu}_r(z) |e_k(z)|^2u(z)dA(z)\\
	 	&\simeq&\int_{\C^n}\widetilde{\mu}_2(z) |e_k(z)|^2u(z)dA(z),\\
	 \end{eqnarray*}
	 
	 Then, using Jensen's formula again, we obtain,
	 \begin{eqnarray*}
	 	\sum_{k=1}^\infty h(C\langle T_\mu e_k,e_k\rangle_{A_u^2})&\lesssim& \int_{\mathbb{C}^n}h(C\widetilde{\mu}_2(z))\left(\sum_{k\geq 1}|e_k(z)|^2\right)u(z)dA(z)\\
	 	&=&\int_{\mathbb{C}^n}h(C\widetilde{\mu}_2(z))\|K_z\|_{A_u^2}^2u(z)dA(z)\\
	 	&\simeq&\int_{\mathbb{C}^n}h(C\widetilde{\mu}_2(z))u(\Delta(z,r))^{-1}dA(z)<\infty,
	 \end{eqnarray*}
	 which implies $T_\mu\in S_h(A_u^2)$.
\end{proof}

Taking $h(t)=t^p$ in Theorem \ref{sc}, we can get the following corollary which characterize Schatten class of Toeplitz operators on $A_u^2$.
\begin{cor}
	Let $1\leq p<\infty$ and $\mu$ a positive Borel measure on $\mathbb{C}^n$ such that the Toeplitz operator $T_\mu:A_u^2\to A_u^2$ is compact. Then $T_\mu\in S_p(A_u^2)$ if and only if
	\begin{eqnarray*}
		\widetilde{\mu}_2(z)\left(u(\Delta(z,r))^{-1}\right)^{1/p}\in A_u^p.
	\end{eqnarray*}
\end{cor}

\end{document}